\documentclass[12pt]{amsart}

\usepackage{amsfonts,amssymb,amsmath,amsthm,mathtools,mathabx,bm, leftindex}
\usepackage{url}
\usepackage{enumerate}
\usepackage{stmaryrd}
\usepackage{bbm}

\theoremstyle{plain}
\newtheorem{theorem}{Theorem}[section]
\newtheorem*{theorem*}{Theorem}
\newtheorem{lemma}[theorem]{Lemma}
\newtheorem{corollary}[theorem]{\bf Corollary}
\newtheorem{remark}[theorem]{\bf Remark}
\newtheorem{example}[theorem]{\bf Example}
\newtheorem{proposition}[theorem]{\bf Proposition}

\newcommand \Val{{\rm{CoV}}}

\newcommand \Mod{\rm{Mod}}
\newcommand \tp{\rm{tp}}

\newcommand \sqsp{\sqsupset}

\newcommand{\sqni}{\mathrel{\vphantom{\sqsupset}\text{\mathsurround=0pt\ooalign{$\sqsupset$\cr$-$\cr}}}}

\newcommand {\cl}[1]{{\rm{cl}(#1)}}

\newcommand \Sp{{\mathsf{S}}}

\newcommand \fra{Fra\"{i}ss\'{e}}
\newcommand \uhr{\upharpoonright}
\newcommand \fr{\scriptscriptstyle{f}}
\newcommand \re{\scriptstyle{r}}

\newcommand \BBf{\mathbf{B}}

\newcommand \FFf{\mathbf{F}}

\newcommand \IIf{\mathbf{I}}
\newcommand \KKf{\mathbf{K}}
\newcommand \LLf{\mathbf{L}}
\newcommand \PPf{\mathbf{P}}
\newcommand \SSf{\mathbf{S}}

\newcommand \KKb{\mathbb{K}}
\newcommand \NN{\mathbb{N}}

\newcommand \PPb{\mathbb{P}}
\newcommand \RRb{\mathbb{R}}

\DeclareMathOperator{\Cov}{\mathbf{Cov}}

\mathtoolsset{showonlyrefs} 

\title{A logic of co-valuations}

\author[M. Malicki]{Maciej Malicki}

\address{Faculty of Mathematics, Informatics and Mechanics of the University of Warsaw, ul. Banacha 2, Warsaw, Poland}
\email{mmalicki@mimuw.edu.pl}
\date{\today}

\thanks{Research was partially supported by the National Science Centre, Poland under the Weave-UNISONO call in the Weave programme [grant no 2021/03/Y/ST1/00072].}

\begin{document}
	\maketitle
	
\begin{abstract}
A co-valuation is, essentially, a minimal finite cover. We introduce a logic based on co-valuations, which play the role of valuations of free variables in classical first-order logic, and show that the fundamental tools of model theory -- such as ultraproducts, compactness, and omitting types -- can be developed in this setup. Using a recently discovered duality between certain countable posets and second-countable compact $T_1$ spaces, we show that these spaces are counterparts of countable universes in first-order logic. Thus, although no topology appears in the initial formulation, the logic of co-valuations turns out to be naturally suited for studying compact topological objects. Standard topological notions, such as connectedness and covering dimension, are easily expressible, and model-theoretic properties, such as atomicity, can be effectively analyzed. The framework also interacts well with \fra-type constructions.
\end{abstract}
	
	\newcommand{\refiningvia}{\pi_1``}

\section{Introduction}
A model-theoretic structure is a set equipped with relations and functions. Despite potential complexity, such structures are trivial from a topological perspective: the only topology they are naturally equipped with is the discrete one. To incorporate topology into logic, two main strategies have been considered.

The first apporach  is point-free topology, which ignores the points of the underlying space and focuses on algebraic structures derived from the topology, such as collections of open sets with operations like union, intersection, and closure. Early examples include Tarski’s work \cite{Ta}, which interprets intuitionistic logic via Heyting algebras of open sets; Tarski also studied algebras of closed sets and, together with MacKinsey \cite{MKTa}, closure algebras. Later, Henson et al. \cite{Heatal} investigated lattices of closed sets of topological spaces.

The second one retains the points but introduces an additional sort to capture the topology. Examples include the works of Grzegorczyk \cite{Gr}, Gurevich \cite{Gu}, and Flum-Ziegler \cite{FlZi}, where suitable fragments of second-order monadic logic are used. An entirely different direction is to study definable topologies, as proposed by Pillay \cite{Pi}.

Both strategies face inherent limitations. Algebraic structures derived from topologies are generally not axiomatizable, and two-sorted structures, where one sort represents the topology, are similarly resistant to axiomatization. As a consequence, the resulting logics often lack desirable properties: they may fail compactness, the downward L\"owenheim-Skolem theorem does not hold, etc.

One partial remedy is to restrict attention to subfamilies of topologies, such as topological bases. However, this is not a complete solution. A single space may admit multiple bases that are not elementarily equivalent. Expressibility is also limited: if one considers a sort for a basis, the downward L\"owenheim-Skolem theorem holds, but only yields countable topological spaces. This severely restricts the ability to express fundamental topological properties such as connectedness or dimension -- for instance, every countable metrizable space is homeomorphic to a subset of the rational numbers.

 One can also focus on specific classes of spaces or alternative structures derived from a topological space. Bankston \cite{Ba} considered regular compact spaces, which are equivalent, via Wallman duality \cite{Wa}, to disjunctive normal lattices.  Nevertheless, the choice of a “right” basis remains problematic, and expressibility continues to depend heavily on the basis. On the other hand, Henson et al. \cite{Heatal} investigated also the ring $C(X)$ of real-valued functions on a topological space $X$. Similarly, Eagle et al. \cite{EaGoVi} associate a Hausdorff compact space $X$ with its $C^*$-algebra $C(X)$, and apply continuous logic. This is a promising research direction, however it seems restricted to spaces without any additional structure: what would be the $C^*$-algebra associated with a compact graph?

In this study, we take a different approach. The starting point is the notion of a co-valaution on a set $M$, which plays the role of a valuation in first-order logic -- that is, an interpretation of free variables. Specifically, a co-valuation is a relation $v \subseteq X \times M$, where $X$ is a finite set (called a context) such that, for $[x]_v=\{a \in M: (x,a) \in v\}$, the family $[v]=\{[x]_v: x \in X\}$ is a minimal  cover of $M$. A fundamental relation between co-valuations is a consolidation: a covaluation $v \subseteq X \times Y$ is a consolidation of $w \subseteq Y \times M$ -- and $w$ is a fragmentation of $v$ -- if $[v]$ is refined by $[w]$, and every set $[x]_v$ is a union of elements of the form $[y]_w$. The associated consolidation  pattern is the relation $\geq \subseteq X \times Y$, defined by $x \geq y$ iff $[x]_v \supseteq [y]_w$, which describes how the consolidation is carried out. A structure $(M,V)$ in the logic of co-valuations is an ordinary first-order structure $M$ together with a family $V$ of co-valuations on the universe of $M$ in contexts drawn from a fixed set of variables,  satisfying certain additional requirements. 
Co-valuations in $V$ are called admissible.

The (relational) language of this logic of co-valuationsis is the same as in first-order logic, except that quantifiers over contexts must explicitly specify their required consolidation pattern. The satisfaction relation is defined as follows. For a structure $(M,V)$, admissible co-valuation $v \subseteq X \times M$, relation symbol $R$, and $x_1, \ldots, x_m \in X$, $M \models_v R(x_1, \ldots, x_m)$ if there exist witnesses $a_i \in [x_i]_v$ such that $R^M(a_1, \ldots, a_m)$. On the other hand, $M \models_v \exists^{\fr}_\geq \phi$ if there exists an admissible $w$ that fragments $v$ with pattern $\geq$, and $M \models_w \phi$. The connectives $\neg$ and $\wedge$ are interpreted as in first-order logic.

This logic behaves well with respect to ultraproducts: it satisfies \L{}o\'s's theorem (Theorem \ref{th:LS}), and therefore enjoys compactness (Corollary \ref{co:Compactness}). Moreover, with an aid of a duality recently discovered by Bartoš, Bice, and Vignati \cite{BaBiVi} between certain countable posets, called $\omega$-posets, and second-countable compact $T_1$ spaces, for every structure, there exists an elementarily equivalent one whose admissible co-valuations yield a second-countable  compact $T_1$ topology. We call such structures compact $\omega$-structures. Thus, a topological version of the downward L\" owneheim-Skolem theorem holds (Theorem \ref{th:LS} ), even though no topology is present in the original setup. Furthemore, the omitting types theorem can be proved (Theorem \ref{th:OM}), using Baire-category methods in the Polish space $\Mod(L)$ of all compact $\omega$-structures in signature $L$ endowed with the logic topology. This leads to a natural notion of atomicity for compact $\omega$-structures (Theorem \ref{th:UniqueAtomic}).
 
One can show that every compact $T_1$ space equipped with closed relations can be naturally presented as a structure in our logic of co-valuations (Corollary \ref{co:CompactAreStructures}). Since admissible co-valuations form only a basis for the underlying topology, not all topological properties are expressible in the logic. Nevertheless, because in fact we work with covers coming from this basis, rather than the basis itself, notions such as connectedeness, perfectness, or covering dimension are expressible; some others, such as metrizability or chainability follow from omitting a suitable family of types (Section 4.4).  

The fact that admissible co-valuations give only a basis is a limitation. However, for many classical topological objects there is a canonical choice of admissible co-valuations. Specifically, this occurs for second-countable compact $T_1$ spaces (possibly with additional relations) that arise as \fra \ limits in a variant of \fra \ theory suitable for this setting (Section 6). When admissible co-valuations are obtained via a sufficiently nice ($L$-faithful)  \fra \  sequence in an appropriate category, the resulting structure is an atomic model of its theory (Theorem \ref{th:FraAtomic}), and any two such structures are isomorphic (Theorem \ref{th:FraAreIso}). This applies, for example, to the arc, the ordered arc, or the pseudo-arc. Thus, for compact objects arising as \fra \ limits, there is a canonical way to select the “correct’’ structure that presents it in the logic of co-valuations.


The paper is organized as follows. Section 3 provides a brief overview of the Barto\v s-Bice-Vignati duality between $\omega$-posets and second- countable compact $T_1$ spaces as well as some other background facts required later. In Section 4, we introduce the logic of co-valuations, define compact structures, and examine several basic topological properties that are expressible in this logic. Section 5 develops the model theory: ultraproducts, Łoś's theorem, the downward L\"owenheim-Skolem theorem, omitting types, and atomic models. Section 6 presents a variant of \fra \ theory adapted to the logic of co-valuations, and discusses atomicity of \fra \ limits along with several examples.

\section{Relations}	
Let $R \subseteq X \times Y$, $S \subseteq Y \times Z$ be relations. The composition $R \circ S \subseteq X \times Z$, and inverse $R^{-1} \subseteq Y \times X$ are defined as $$R \circ S =\{ (x,z) \in X \times Z: \exists y (x \, R \, y \, S \, z)\}, \ R^{-1}=\{(y,x) \in (Y,X): xRy\}.$$ 
%

For $x \in X$, and $A\subseteq X$,  $$[x]_R = \{y \in Y : xRy \}, \, [R]_A=\{[x]_R: x \in A \};$$  we shortly write $[R]=[R]_X$. We think of $X$ as the co-domain, and $Y$ as the domain of $R$. Thus, $R$ is \emph{surjective} if for every $x \in X$ there is $y \in Y$ such that $xRy$. It is \emph{co-surjective} if $R^{-1}$ is surjective. The relation $\pi_R \subseteq X \times Y$ is defined as the largest $R' \subseteq R$ such that  
$$ (xR'y \mbox{ and }  x'Ry) \rightarrow x=x'.$$
%
$R$ is \emph{co-injective} if $\pi_R$ is surjective. It is \emph{co-bijective} if it is co-injective and co-surjective.

\begin{remark}
	\label{re:CoInjSurj}
	If $R$ is co-injective, then it is surjective. It is co-bijective iff $[R]$ is a minimal cover of $Y$.
\end{remark} 

For $A \subseteq X$,  $$A^R=\{y \in Y: \exists a \in A \ a R y  \}.$$
We say that $A$ is \emph{$R$-refined} by $B \subseteq Y$ (or that $A$ is refined by $B$, provided that $R$ is clear from the context) if $A^R \supseteq B$. Typically, $R \subseteq \PPb \times \PPb$ is  an ordering relation $\geq$ of a poset $\PPb$, e.g., the inclusion relation $\supseteq$.  In the case that the ordering relation of a poset $\PPb$ is not explicitely specified, we will refer to it by $\leq$ (or by $\geq$). For $A,B \subseteq \PPb$, we say that $A$ is an \emph{$\leq$-up-set} (or up-set) if $A^\leq=A$,

\section{Constructing compact spaces from posets}
\label{se:Constructing}
	
	
In this section, we present a brief overview of the  Barto\v s–Bice–Vignati duality connecting certain countable posets and second-countable compact $T_1$  spaces in the spirit of Stone duality. It has  been developed in \cite{BaBiVi}, to which the reader is referred for a detailed exposition. We provide some proofs to illustrate the role played by the main concepts of this theory.

We begin with some motivation. Let $X$ be a second-countable compact $T_1$ space, and let $\PPb_n$, $n \in \NN$, be a sequence  of finite open covers of $X$ such that every $\PPb_{n}$ is refined by $\PPb_{n+1}$, and every open cover of $X$ is refined by some $\PPb_n$. In particular, $\bigcup \PPb_n$ is a basis of $X$ (see Proposition \ref{pr:SeqGivesBasis}).  For $T_2$, i.e., metrizable, compact spaces, the construction of such a sequence, using the Lebesgue number of a finite cover, is straightfoward; the general case is also elementary (see \cite[Proposition 1.11]{BaBiVi}). Let the disjoint union $\PPb= \bigsqcup \PPb_n$ be ordered by inclusion, i.e., $p \geq q$ iff $p \supseteq q$, where $p \in \PPb_m$, $q \in \PPb_n$, $m \leq n$.   

For every $x \in X$, let us consider the family $S \subseteq \PPb$ of all those $p \in \PPb$ that contain $x$. Such families can be characterized among all subsets of $\PPb$  as precisely those whose complements $T=\PPb \setminus S$ are maximal families in $\PPb$ that do not cover $X$.  Equivalently -- since $T$ fails to cover $X$ if and only if none of its subsets covers  $X$
-- the set $S$ can be described as a minimal family that intersects every cover of  $X$ consisting of elements of  $\PPb$. In the terminology of \cite{BaBiVi}, such an $S$ is called a \emph{minimal selector}. A key point of this description is that it is purely order-theoretic: the very construction of $\PPb_n$'s ensures that a subset of $\PPb$ is a cover of $X$ if and only if it is $\supseteq$-refined by some $\PPb_n$.

In this setting we can define the analogue of the Stone space of a Boolean algebra. The \emph{spectrum} $\Sp \PPb$ of the poset $\PPb$ is the family of all minimal selectors in $\PPb$. One verifies that the sets $p^\in$ of all minimal selectors that contain $p \in \PPb$, form a basis for a topology on $\Sp \PPb$, and that this topology is homeomorphic to the original topology on $X$.

This approach can be developed abstractly for a class of posets that capture the structure of the sequence of finite open covers described above. A poset $\PPb$ is called an \emph{$\omega$-poset} if there exists a (necessarily unique) rank function  $r:\PPb \to \NN$ assigning $0$ to maximal elements, and such that all the \emph{levels} $\PPb_n=r^{-1}(n)$, $n \in \NN$, are finite.  Clearly, $\PPb=\bigsqcup \PPb_n$.

\begin{remark}
In \cite{BaBiVi}, levels consist of all minimal elements in the set of all elements of rank at most $n$. However, for $\omega$-posets that will be considered in this paper, these two definitions are equivalent.
\end{remark}


As a matter of fact, we will work with  $\omega$-posets of a particularly simple form, where the ordering relation is determined by relations between consequtive levels. Let $\PPb_n$, $n \in \NN$, be finite sets, and let $\geq_n \subseteq \PPb_n \times \PPb_{n+1}$, $n \in \NN$, be co-surjective relations. 
Let $\PPb=\bigsqcup \PPb_n$, and let $\geq$ be the reflexive and transitive closure of $\bigsqcup \geq_n$.  It is straightforward to verify that the resulting poset  $(\PPb,\geq)$ is an $\omega$-poset, referred to as an \emph{$\omega$-chain} $(\PPb_n, \geq_n)$. For an $\omega$-chain $(\PPb, \geq)$, $\PPb_n$ always denote the levels of $\PPb$, $\geq^m_n=\geq \uhr \PPb_m \times \PPb_{n}$, where $m \leq n$, and $\geq_n=\geq^n_{n+1}$. We call $(\PPb_n, \geq_n)$ \emph{level-injective} if all $\geq_n$ are co-injective.

\begin{remark}
	In the terminology of \cite{BaBiVi}, $\omega$-chains are graded, predetermined $\omega$-posets.  
\end{remark}

\begin{remark}
	Note that $(\PPb_n, \geq_n)$ has no minimal elements  iff all $\geq_n$ are  surjective. In particular, this holds  when all $\geq_n$ are co-injective, i.e.,   when $(\PPb_n, \geq_n)$  is level-injective.
\end{remark}

Let $A, B$ be covers of a set $X$ such that $A$ is refined by $B$. We say that $A$ is a \emph{consolidation}  of $B$ (or that $A$ \emph{consolidates} $B$), and that $B$ is a \emph{fragmentation}  of $A$ (or that $B$ \emph{fragments} $A$),  if every element of $A$ is a union of elements of $B$, i.e., for $a \in A$,  $$a=\bigcup \{b \in B: b \subseteq a \}.$$  
 
\begin{proposition}
\label{pr:MinCoversCoInj}
Let $A$, $B$ be minimal covers of a set $X$ such that $A$ is refined by $B$. Then $\supseteq \uhr A \times B$ is co-injective. 
\end{proposition}

\begin{proof}
Fix $p \in A$, and, using minimality of $A$, $x \in  p \setminus \bigcup \{p' \in A: p' \neq p \}$. As  $B$ covers $X$, there is $q \in B$  with $x \in q$. As it refines $A$, there is $p' \in A$ such that $p' \supseteq q$. Clearly, $p' \supseteq q$ implies $p'=p$ for every $p' \in A$.
\end{proof}

\begin{proposition}
\label{pr:MinCoversChains}
Let $\PPb_n$, $n \in \NN$, be minimal covers of a set $X$ such that each $\PPb_n$ consolidates $\PPb_{n+1}$. Then $(\PPb,\geq)=(\PPb_n,\supseteq \uhr \PPb_n \times \PPb_{n+1})$  is a level-injective $\omega$-chain, and, for any $p \in \PPb_m$, and $q \in \PPb_n$, where $m \leq n$, we have
\[ p \geq q \iff p \supseteq q.\]
Thus, slightly abusing notation, we can denote $(\PPb,\geq)$ by $(\bigsqcup \PPb_n, \supseteq)$. 
\end{proposition}

\begin{proof}
Level-injectivitiy follows from Proposition \ref{pr:MinCoversCoInj}. Let $\supseteq_n=\supseteq \uhr \PPb_n \times \PPb_{n+1}$, and, for $m<n$, denote by $\supseteq^m_n$ appropriate compositions of $\supseteq_n$. We claim that if $\supseteq_n^0$ is the transitive closure of $\bigcup_{m \leq  n } \supseteq_m$, then $\supseteq_{n+1}^0$ is the transitive closure of $\bigcup_{m \leq n+1 } \supseteq_m$. Fix $m \in \NN$ with $m<n$, and $p \in \PPb_m$, $q \in \PPb_{n+1}$ with $p \supseteq q$. By minimality of the cover $\PPb_{n+1}$, there is $x \in q \setminus \bigcup \{q' \in \PPb_{n+1}: q' \neq q \}$. Because $\PPb_m$ consolidates $\PPb_n$, there is $q' \in \PPb_n$ such that $x \in q' \subseteq p$. But then we must have $q \subseteq q'$, which immediately implies the claim.
\end{proof}

\begin{theorem}[Theorem 1.34 in \cite{BaBiVi}]
	\label{th:CoinitialCapBasis}
	Let $X$ be a compact $T_1$ space. For any minimal open covers $A, B$ of $X$ there is a minimal open cover $C$ of $X$ that fragments both $A$ and $B$. In particular,  if $X$ is second-countable,  then there exist minimal open covers $\PPb_n$ of $X$, $n \in \NN$, such that every $\PPb_{n}$ consolidates $\PPb_{n+1}$, and every open cover of $X$ is refined by some $\PPb_n$. 
\end{theorem}

\begin{proof}
	For each $x \in X$, let $d_x = \bigcap \{a \in  A \cup B : x \in  a\}$. As $A$ and $B$ are finite, so is $D= \{d_x : x \in  X\}$. For each $d \in  D$, choose some $x_d \in X$ such that $d= d_{x_d}$ and denote the set of all the other chosen points by $$f_d = \bigcup_{e \in D \setminus \{d\}} \{x_e\}.$$ We then have a minimal open cover refining both $A$ and $B$ given by $$C= \{d \setminus  f_d : d \in D\}.$$ Also note that if $y \in  a \in A$ then $y  \neq  x_d$ for any $d \neq d_y$ (because $y= x_d$ implies $d_y = d_{x_d}
	= d$), so $y \in d_y \setminus f_{d_y} \subseteq  a$. This shows that $a= 
	\bigcup \{c \in C: c \subseteq a\}$, for all $a \in A$, i.e., $C$ fragments $A$. For the same reason, $C$ fragments $B$.  
\end{proof}

\begin{proposition}
	\label{pr:SeqGivesBasis}
	Let $\mathcal{C}$ be a family of open covers of a  compact $T_1$ space $X$ such that every open cover of $X$ is refined by some $C \in \mathcal{C}$. Then $\bigcup \mathcal{C}$ is a basis of $X$.
\end{proposition}

\begin{proof}
	Fix an open $c \subseteq X$, and $x \in c$. Let $C$ be an open cover of $X$ extending $\{c\}$, and such that $x \not \in \bigcup C \setminus \{c\}$. Fix $C' \in \mathcal{C}$ that refines $C$, and let $c' \in C'$ be such that $x \in c'$. Then we must have $c' \subseteq c$. 	
\end{proof}

Let $\PPb$ be an $\omega$-poset. A finite subset $B \subseteq \PPb$ is called a \emph{band} if every $p \in \PPb$ is comparable with some $b \in B$.  A subset $C \subseteq \PPb$ is called a \emph{cap} if it is  $\geq$-refined by a band. 
We call $S \subseteq \PPb$ a \emph{selector} if it has non-empty intersection with every cap (equivalently: $\PPb \setminus S$ is not a cap).  The \emph{spectrum} $\Sp \PPb$ of $\PPb$ is the family of all minimal selectors in $\PPb$. For $p \in \PPb$, let $$p^\in=\{S \in \Sp \PPb: p \in S \},$$ and, for $A \subseteq \PPb$, let $$A^\in=\{p^\in: p \in A \}.$$ We equip $\Sp \PPb$ with the topology $\tau_\PPb$ generated by the sub-basis $\PPb^\in$.

\begin{remark}
\label{re:LevelsEnough}
For $\omega$-chains $\PPb$ without minimal elements, e.g., level-injective $\omega$-chains, levels are bands, and caps can be equivalently defined as subsets of $\PPb$ that are refined by a level.
\end{remark}

	\begin{lemma}[Propositions 2.2, 2.3, 2.4, and 2.13 in \cite{BaBiVi}]
		\label{le:SelectorFilterAndContainsMin}
		Let $\PPb$ be an  $\omega$-poset.
		\begin{enumerate}
			\item Every selector contains a minimal selector.
			\item Every minimal selector is a filter, i.e., it is an up-set $S$ such that for any $p,q \in S$ there is $r \in S$ with $r \leq p,q$.
		\end{enumerate}
	\end{lemma} 
	
\begin{proof}
To prove (1), we use Zorn's lemma. Let $S_i$, $i \in I$, be an infinite chain (under inclusion) of selectors. We show that $S= \bigcap_i S_i$ is a selector. Let $C$ be a cap. We can assume that $C$ is finite, so there is $c \in C$ such that for $S_i \cap C=\{c\}$ for almost all $i \in I$. But $S_i$'s form a chain, so, actually, $S_i \cap C=\{c\}$ for all $i \in I$.

To prove (2), fix a minimal selector $S$. Minimality means that, for every $s \in  S$, there is a cap $C$ such that $S \cap  C= \{s\}$. Note that, for any $p \geq s$, $(C \setminus \{s\}) \cup \{p\}$ is refined by $C$ and is thus also a cap. As $S$ must also overlap this new cap, the only possibility is that $S$ also contains $p$. Thus, $S$ is an up-set.  

Fix now $p,q \in S$. Again, by minimality of $S$, there are caps $C$, $D$ such that $C \cap S= \{q\}$, and $D\cap S= \{r\}$. We can assume that $C$ and $D$ are finite. Then there exists $n \in \NN$ such that the band $B$ consisting of all minimal elements in the set of elements of rank at most $n$, refines both $C$ and $D$. As $B$ is a cap, there is $r \in S \cap B$. But $B$ refines $C$ and $D$, so there are $c \in C$, $d \in D$ such that $r \leq c,d$. Because $S$ is an up-set, we also have $c,d \in S$. Thus, $c=p$, $d=q$, i.e., $S$ is a filter.
\end{proof}
	
	\begin{corollary}
		$\PPb^\in$ is a basis of $\tau_\PPb$.
	\end{corollary}
	
	\begin{proof}
		Suppose that $S \in p^\in \cap q^\in$. As $S$ is a filter, there is $r \in S$ such that $r \leq p,q$. But then $S \in r^\in \subseteq p^\in \cap q^\in$.
	\end{proof}
	
A basis $\PPb$  of a topological space $X$ is called a \emph{cap-basis} if every cover of $X$ consisting of sets from $\PPb$ is a cap in the poset $\PPb$ ordered by inclusion.

\begin{remark}
\label{re:SeqGivesCapBasis}
Let $\PPb_n$, $n \in \NN$, be minimal open covers of a second-countable compact $T_1$ space $X$  such that every $\PPb_{n}$ is refined by $\PPb_{n+1}$, and every finite open cover of $X$ is refined by some $\PPb_n$. Then $\PPb=\bigcup \PPb_n$ is a cap-basis of $X$.
\end{remark}

\begin{proposition}[Proposition 2.8 in \cite{BaBiVi}]
		\label{pr:BBV}
		The topology $\tau_\PPb$ is a second-countable compact $T_1$ topology on $\Sp \PPb$, and $\PPb^\in$ is a cap-basis of $\tau_\PPb$.
\end{proposition}
	
	\begin{proof}
		Given any distinct $S, T \in \Sp \PPb$, minimality implies that there exist $s \in S \setminus T$, and $t \in T \setminus S$. This means that $S \in s^\in$, $T \not \in s^\in$, and $S \not \in t^\in $, $T \in t^\in$, showing that $\Sp \PPb$ is $T_1$.
		
		We prove that $\tau_\PPb$ is a compact topology. Without loss of generality, we can assume that $\Sp \PPb$, and so $\PPb$, is infinite. Fix a cap $C$, and a band $B$ that refines $C$. Clearly, $B^\in$ is a band in $\PPb^\in$, i.e., $C^\in$ is a cap. Fix $S \in \Sp \PPb$. As $S$ is infinite, $B$ is finite, and $\PPb$ is an $\omega$-poset, there must exist $s \in S$, and $b \in B$ such that $s \leq b$. Because $B$ refines $C$, there is $c \in C$ such that $b \leq c$. But $S$ is an up-set, so $c \in S$, i.e., $S \in c^\in$. As $S$ was arbitrary, we get that $C^\in$ covers $\Sp \PPb$.
		
		Suppose now that $C \subseteq \PPb$ is not a cap. In other words, $\PPb \setminus C$ is a selector, so it contains a minimal selector $S$. But then $c \not \in S$, i.e., $S \not \in c^\in$ for every $c \in C$, and $C^\in$ does not cover $\Sp \PPb$. Therefore covers of $X$ by sets from $\PPb^\in$ are given by caps in $\PPb$. In particular, $\PPb^\in$ is a cap-basis, and
		$\tau_\PPb$ is compact.
	\end{proof}

\begin{proposition}
\label{pr:LevelMin-CovMin}
	If $\PPb$ is a level-injective $\omega$-chain, then every level $\PPb_n^\in$ is a minimal $\tau_\PPb$-open cover of $\Sp \PPb$, and every $\tau_\PPb$-open cover is refined by some $\PPb_n^\in$. In particular, $p^\in \neq \emptyset$, $p \in \PPb$.
\end{proposition}

\begin{proof}
The facts that  every $\PPb_n^\in$ is a $\tau_\PPb$-open cover of $\Sp \PPb$, and every $\tau_\PPb$-open cover is refined by some $\PPb_n^\in$, directly follow from Remark \ref{re:LevelsEnough} and Proposition \ref{pr:BBV}. It remains to show minimality of  $\PPb_n^\in$'s.

Fix $n_0 \in \NN$, and $p \in \PPb_{n_0}$. As $\PPb_{n_0}$ is a band, $\PPb^\in_{n_0}$ is a cover of $\Sp \PPb$.  Using co-injectivity of $\geq_n$'s, we construct $p_n	\in \PPb_{n}$, $n \geq n_0$, so that $p_{n_0}=p$, and, for every $p' \in \PPb_{n}$, $p_{n+1} \geq p'$ implies $p'=p_n$. Observe that since $\geq$ is the transitive and reflexive closure of $\geq_n$, for every $p' \in \PPb_{n_0}$, $p_{n+1} \geq p'$ implies $p'=p$. Clearly, $\{p_n\}^\leq$ is a selector, and by our construction, for every minimal selector $S \subseteq \{p_n\}^\leq$, and $p' \in \PPb_{n_0}$, $p' \in S$ implies $p'=p$. In particular, $S \in p^\in \setminus \bigcup_{ \{q \in \PPb_{n_0}:q \neq p\} } q^\in$ .
\end{proof}

\begin{proposition}[Proposition 2.9 in \cite{BaBiVi}]
\label{pr:HomeoSpectrum}
Let $X$ be a second-countable compact $T_1$ space $X$, and let $\PPb$ be a cap-basis of $X$ ordered by inclusion. Then $x \mapsto x^\in=\{p \in \PPb: x \in p \}$ is a homeomorphism between $X$ and $\Sp \PPb$. In particular, every second-countable compact $T_1$ space arises as the spectrum of an $\omega$-chain.
\end{proposition}

\begin{proof}
First we observe, arguing as in the introduction to Section 3, that every $x^\in$ is indeed a minimal selector in $\PPb$. As $X$ is $T_1$, for any distinct $x, x' \in X$, there is $p \in \PPb$ such that $x \in p$, $x' \not \in p$, i.e., the mapping $x \mapsto x^\in$ is injective. Finally, we note that, for every $p \in \PPb$, $x \in p $ iff $x^\in \in p^\in$, which implies that it is continuous, and its image is dense in $\Sp \PPb$. Thus, it is onto $\Sp \PPb$, i.e., a homeomorphism between $X$ and $\Sp \PPb$. 
\end{proof}

\begin{remark}
\label{re:pGoesTop^in}
Note that $x \in p$ iff $x^\in \in p$, so the homeomorphism $x \mapsto x^\in$ maps $p \in \PPb$ onto $p^\in$.
\end{remark}



%
\begin{proposition}
\label{le:GoodBehavior}
Let  $\PPb$ be a level-injective $\omega$-chain. Then $x \geq y$ iff $x^\in \supseteq y^\in$, for $x \in \PPb_m$, $y \in \PPb_n$, where $m \leq n$. 
\end{proposition}

\begin{proof}
	The implication from left to right is obvious. To see the other implication, we fix $x \in \PPb_{m_0}$, $y \in \PPb_{n_0}$ such that $m_0 \leq n_0$, and $x \not \geq y$. Then we proceed as in the proof of Proposition \ref{pr:LevelMin-CovMin}, i.e., we construct $y_n	\in \PPb_{n}$, $n \geq n_0$, so that $y_{m_0}=y$, and, for every $y' \in X_{n}$, $y' \geq y_{n+1}$ implies $y'=y_n$. For every minimal selector $S \subseteq \{y_n\}^\leq$, we have $S \in y^\in \setminus x^\in$, i.e., $x^\in \not \supseteq y^\in$.
\end{proof}

\begin{corollary}
\label{co:IsomorphicPosets}	
Let $(\PPb, \geq)$ be a  level-injective $\omega$-chain. Then $p \mapsto p^\in$ is an order-isomorphism between $(\PPb, \geq)$ and the level-injective $\omega$-chain  $(\bigsqcup \PPb_n^\in, \supseteq)$.
\end{corollary}

Let $A$, $B$ be covers of a set $X$, and let $a \in A$, $b \in B$. We define $b \lhd_B a$ if, for $b' \in B$, $b' \cap b \neq \emptyset$ implies $b' \subseteq a$; we write $A \rhd B$, when  the relation $\rhd_B \subseteq A \times B$ is surjective, i.e., $A$ is $\rhd_B$-refined by $B$.
\begin{proposition}
\label{pr:StarRefClosure}
Let $A$, $B$ be open covers of a topological space $X$, and let $a \in A$, $b \in B$. If $b \lhd_B a$, then $\cl{b} \subseteq a$
\end{proposition}
	
	\begin{proof}
	Fix $x \in \cl{b}$. As $B$ is a cover of $X$, there is $b' \in B$ such that $x \in b'$. But then $b' \cap b \neq \emptyset$, so $x \in b' \subseteq  a$.
	\end{proof}

\begin{proposition}
\label{pr:Regular}
Let $X$ be a compact $T_1$ space.  Assume that there exists a family $\mathcal{C}$ of open covers of $X$ such that every open cover of $X$ is refined by some $C \in \mathcal{C}$, and every $C \in \mathcal{C}$ is $\rhd$-refined by some $D \in \mathcal{C}$. Then $X$ is regular.
\end{proposition}

\begin{proof}
Fix $C \in \mathcal{C}$, $c \in C$ and $x \in c$. Let $C'$ be an open cover of $X$ that extends $\{c\}$, and is such that $x \not \in \bigcup \{c' \in C': c' \neq c\}$. Fix $C'' \in \mathcal{C}$ that refines $C'$. Finally, fix $D \in \mathcal{C}$ with $D \lhd C''$, and $d \in D$ such that $x \in d$. There is $c'' \in C''$ such that $d \lhd_D c''$. But $x \in c''$, and $C''$ refines $C'$, so $c'' \subseteq c$, and $x \in d \lhd_D c$. By Proposition \ref{pr:StarRefClosure}, $\cl{d} \subseteq c$.
\end{proof}	

Let $\PPb$ be an $\omega$-poset.  For a cap $C$ in $\PPb$, and $p \in \PPb$, the \emph{star} $Cp$ of $p$ is  defined as $$Cp=\{c \in C: \exists r \in \PPb (r \leq c,p) \}.$$ For $p, q \in \PPb$, we define $q \lhd_C p$ iff $r \leq p$ for every $r \in Cq$. We call $\PPb$ \emph{regular} if every cap  $C$ is $\rhd_D$-refined by some cap $D$. 

\begin{corollary}[Corollary 2.40 in \cite{BaBiVi}]
\label{co:Regular}
Let $\PPb$ be an $\omega$-poset. If $\PPb$ is regular, then $\Sp \PPb$ is regular, i.e., metrizable. 
\end{corollary}

\begin{remark}
\label{re:MetrizablePosetIsRegular}
Using the Lebesgue number, can easily prove that if $X$ is second-countable compact metrizable space, then for every sequence of minimal covers $\PPb_n$ of $X$, $n \in \NN$, such that every $\PPb_n$ consolidates  $\PPb_{n+1}$, and every open cover of $X$ is refined by some  $\PPb_n$, the $\omega$-chain $(\bigsqcup \PPb_n, \supseteq)$ is regular.
\end{remark}

		


\section{A logic of co-valuations}
	
	\subsection{Co-valuations, refinements, and consolidations.} 
	%
	
	%
	
	
 A co-bijective relation $v \subseteq X \times Y$ is called a \emph{co-valuation} on $Y$ in $X$. The family of all co-valuations on $Y$ in $X$ is denoted by $\Val(X, Y)$.  For $v \in \Val(X,Y)$, the set $X$ will be also referred to as $X_v$. 

\begin{remark}
\label{re:Identifycotext}
Recall that, for $v \in \Val(X,Y)$, the family $[v]$ is a minimal cover of $Y$. As it will usually not cause any ambiguity, we may sometimes identify $[v]$ and $X_v$.
\end{remark}

Let $v \in \Val(X, M)$,  $w \in \Val(Y, M)$, and let $\geq \Val(X,Y)$.  We write $v \geq w$ if $[v]$ is refined by $[w]$ with pattern $\geq$, i.e., $x \ge y$ implies $[x]_v \supseteq [y]_w$. The relation $\geq^v_w \subseteq X \times Y$ is defined by
$$x \geq^v_w y \Leftrightarrow [x]_v \supseteq [y]_w.$$
We denote  $C_\geq(w)=\geq \circ w$, and call $C_\geq(w)$ a \emph{consolidation} of $w$ \emph{with pattern} $\geq$, and $w$ a \emph{fragmentation} of $C_\geq(w)$ \emph{with pattern} $\geq$. Also, $v$ is a consolidation of $w$ (and $w$ is a fragmentation of $v$) if there exists a pattern $\geq \in \Val(X,Y)$ such that $v=C_{\geq} (w)$.  
Note that $v$ is a consolidation of $w$ iff the cover $[v]$ of $M$ is a consolidation of the cover $[w]$ as defined in Section \ref{se:Constructing}. If $v$ is a consolidation of $w$, by Proposition \ref{pr:MinCoversCoInj}, we have that $\geq^v_w \in \Val(X,Y)$.


A sequence $(v_i)$, $v_i  \in \Val(X_i,M)$, of co-valuations is called an $\omega$-\emph{chain} if each $v_{i}$ consolidates $v_{i+1}$, i.e., $([v_i], \supseteq \uhr [v_i] \times [v_{i+1}])$ is an $\omega$-chain poset as in Proposition \ref{pr:MinCoversChains}.

\begin{remark}	
\label{re:CovertoVal}
	Every minimal finite cover $\PPb$ of a set $M$  gives rise to a co-valuation $v \in \Val(\PPb,M)$, where $(p,a) \in v$ iff $a \in p$. Similarly, a sequence $\PPb_i$, $i \in \NN$, of minimal finite covers of $M$  such that each $\PPb_i$ consolidates $\PPb_{i+1}$, gives rise to an $\omega$-chain of co-valuations. 
\end{remark}
	
	
	%
	%
	
	
\subsection{Structures and interpretations.} Let us select a countably infinite set of variables with a special variable $*$. A \emph{context} is a finite set of variables. In practice, we may regard any finite set as a context.
	
Formulas $\phi(X)$ in context $X$ are defined by the following recursion scheme. If $R$ is a relation symbol of arity $a(R)$ with variables contained in $X$, then
\[R, \, \top \]
are formulas in context $X$.
If $\phi(X)$, $\psi(X)$ are formulas in context $X$, then
\[ \neg\phi(X), \, \phi(X) \land \psi(X) \]
are formulas in context $X$.
Finally, if $\phi(Y)$ is a formula in context $Y$, and $\geq \in \Val(X,Y)$, then
\[ \exists^{\fr}_{\geq} \phi(Y) \]
is a formula in context $X$.

We write $\bot$ for $\neg \top$, $\phi \lor \psi$ for $\neg (\neg \phi \land \neg \psi)$, $\phi \rightarrow \psi$ for $\psi \lor \neg \phi$, $\forall^{\fr}_{\geq} \phi$ for $\neg \exists^{\fr}_\geq \neg \phi$, and $\exists^{\fr}_{X\geq Y} \phi$ to explicitly indicate the domain and co-domain of $\geq$.  Note that $\exists^{\fr}_{X\geq Y} \phi$ has a unique context $X$ (as opposed to, e.g., atomic formulas that can have multiple contexts). A type ($n$-type) $p(X)$ is a set of formulas in a fixed signature and context $X$ (of size $n$). A type $p(X)$ is complete if for every formula $\phi(X)$ either $\phi \in p(X)$ or $\neg \phi \in p(X)$. A sentence is a formula in the context $\bar{*}=\{*\}$, and a theory is a type consisting only of sentences.
	
We will also use the following conventions. If $\geq \in \Val(\bar{*},X)$, we write $\exists^{\fr}_X$ instead of $\exists^{\fr}_{\geq}$ or $\exists^{\fr}_{\bar{*} \geq X}$. Also, we omit $X$ in $\exists^{\fr}_{X\geq Y}$ if it can be unequivocally read off the formula. For example, we write $\forall^{\fr}_X \exists^{\fr}_{\geq Y} \phi$ instead of $\forall^{\fr}_{\succeq} \exists^{\fr}_{\geq} \phi$, where $\succeq \in \Val(\bar{*},X)$, $\geq \in \Val(X,Y)$.

Let $V$ be a family of co-valuations on a set $M$. We write  $V^X$ for $V \cap \Val(X,M)$, $[\bar{x}]_v$ for $\prod_k [x_k]_v$, where $\bar{x}=(x_k)$,  $x_k \in X$. Also, for $v \in \Val(X,M)$, $w \in \Val(Y,M)$, $(x_k) \geq (y_k)$ if $[x_k]_v \supseteq [ y_k]_w$, for all relevant $k$. 
	
By a \emph{proto-structure} in signature $L$,  we mean a triple $M = (M,\{R^M_\alpha\},V)$ consisting of
	
	\begin{enumerate}
		\item a first-order logic structure $(M,\{R^M_\alpha\})$ in a relational signature $L=\{R_\alpha\}$,
		\item a family $V$ of co-valuations on $M$ in contexts that is closed under  consolidations by contexts, i.e., $v \in V^Y$, $\geq \in \Val(X,Y)$ implies $C_{\geq}(v) \in V^X$, where $X$, $Y$ are contexts.
	\end{enumerate}
	Co-valuations in $V$ are sometimes referred to as \emph{admissible} co-valuations. To simplify notation, if it does not cause ambiguity, we may also refer to a proto-structure $(M,\{R^M_\alpha\},V)$ by $(M,V)$ or  by $M$.
	
	For a proto-structure $(M,V)$, and co-valuation $v \in V^X$, the type $\tp(v)$ of $v$ is the set of all formulas $\phi(X)$ in context $X$ such that  $M \models_v \phi$. For proto-structure $(N,W)$ in the same signature as $(M,V)$, a bijection $f:M \to N$ is an \emph{isomorphism} if it preserves relations, and valuations, i.e., for $f[v]=\{(x,f(a)): (x,a) \in v \}$, where $v \in V$, and $f[V]=\{f[v]: v \in V\}$, we have $f[V]=W$.
	
Given a proto-structure $M = (M, V)$, a co-valuation $v \in V^X$, and a formula $\phi(X)$, the satisfaction relation $M \vDash_v \phi$ is defined by the following recursion scheme, where $\bar{x}$ is a tuple in $X$:
	\[ M \vDash_v \top, \]
	\[	M \vDash_v R(\bar{x}) \iff [\bar{x}]_v \cap R^M \neq \emptyset, \]
	\[M \vDash_v \neg \phi \iff M \nvDash_v \phi, \]
	\[	M \vDash_v \phi \land \psi \iff M \vDash_v \phi \mbox{ and }  M \vDash_v \psi, \]
	\[ M \vDash_v \exists^{\fr}_{\geq} \phi \iff
	\exists_{w \in V} (v=C_\geq(w) \mbox{ and } M \vDash_w \phi). \]

\begin{example}
Assume that $(M,V)$ is a proto-structure, $X=\{x_0,x_1\}$ is a context, $R(x_0,x_1)$ is a binary relation symbol, $S(x_1)$ is a unary relation symbol, $v \in V^X$, and $a_0,a_1,a'_1 \in M$.  Then $a_0,a_1,a'_1 \in M$ witness that $$M \vDash_v R(x_0,x_1) \land S(x_1)$$ if $a_0\in [x]_v$, $a_1,a'_1 \in [y]_v$, and $R^M(a_0,a_1)$, $S^M(a'_1)$. Note that it is not required that  $a_1=a'_1$.  
\end{example}

\begin{example}
Assume that $(M,V)$ is a proto-structure, $X=\{x_0,x_1\}$, $Y=\{y_0,y_1,y_2\}$ are contexts, $\geq \in \Val(X,Y)$, and $R(y_0,y_1,y_2)$ is a ternary relation symbol. Consider the formula $\exists^f_\geq R(y_0,y_1,y_2)$ in context $X$, and $v \in V^X$. A co-valuation $w \in V^Y$ witnesses that $$M \vDash_v \exists^f_\geq R(y_0,y_1,y_2)$$ if $v=C_\geq(w)$, and $$M \vDash_w R(y_0,y_1,y_2),$$ i.e., there exist $a_0 \in [y_0]_w$, $a_1 \in [y_1]_w$, $a_2 \in [y_2]_w$ such that  $$R^M(a_0,a_1,a_2).$$  
\end{example}

\begin{remark}
		Let $(M,V)$ be a proto-structure in signature $L$, and let $R \in L$. Let $v \in V^X$, $w \in V^Y$, and $\geq \in \Val(X,Y)$. If $M \vDash_w R(\bar{y})$, and $\bar{x} \in X^{a(R)}$ is such that $\bar{y} \leq \bar{x}$, then $M \vDash_v R(\bar{x})$.  
\end{remark}
\medskip
	
To develop a well-behaved model theory for the logic of co-valuations, we need to work with proto-structures that satisfy certain additional requirements. A \emph{structure} in signature $L$ is a proto-structure $(M,\{R_\alpha\},V)$ in signature $L$ that satisfies the following axioms:
	
	\begin{enumerate}[(I)]
		\item $V$ is directed, i.e., for all $u,v \in V$ there is $w \in V$ that fragments $u$ and $v$,
		\item for all $R \in L$, $v \in V^X$, and $\bar{x} \in X^{a(R)}$ such that $M \vDash_v R(\bar{x})$, there are $w \in V^Y$, $\bar{y} \in Y^{a(R)}$, and $\bar{a} \in [\bar{y}]_{\pi_{w}}$ such that $\bar{x} \geq^v_w \bar{y}$, and $R^M(\bar{a})$.
	\end{enumerate} 
	
We call $w$ as in Axiom (I) a witness of Axiom (I) (for $u$, $v$). Similarly, $w$, $a$ are witnesses of Axiom (II) (for $R$, $v$, $\bar{x}$).

For a structure $(M,V)$, and co-valuation $v \in V^X$, the type $\tp(v)$ of $v$ is the set of all formulas $\phi(X)$ in context $X$ such that  $M \models_v \phi$. We say that a type $p(X)$ is realized in $(M,V)$ if there is $v \in V^X$ such that $M \vDash_v \phi$ for every $\phi \in p$; otherwise, $p$ is omitted by $M$.  A type is \emph{consistent} if there exists a structure $(M,V)$ that realizes it.  A structure that realizes a theory $T$ is called a \emph{model} of $T$.

\begin{remark}
The somewhat technical Axiom (II) is needed for the proof of the downward L\" owenheim-Skolem theorem (Theorem \ref{th:LS}) to work. Let us point out that the seemingly more natural form of Axiom (II), mimicking the requirement that co-valuations give rise to minimal covers: 

\medskip
for all $R \in L$, $v \in V^X$, and $\bar{x} \in X^{a(R)}$ such that $M \vDash_v R(\bar{x})$, there is $\bar{a} \in [\bar{x}]_{\pi_{v}}$ such that $R^M(\bar{a})$, 
\medskip

would be too restrictive for our purposes. For example, suppose that $R$ is a unary relation symbol, $R^M=\{a_0\}$ is a singleton, and $v \in V^X$ is such that there are two distinct $x,x' \in X$ with $a_0 \in [x]_v,[x']_v$.  Then $M \models_v R(x)$ but $a_0 \not \in [x]_{\pi_v}$, i.e., there is no $a \in [x]_{\pi_v}$ such that $R^M(a)$.  
\end{remark}
	
\begin{proposition}
\label{pr:Amalgamation}
Let $(M,V)$ be a structure. Then $V$ has amalgamation in the following sense. Let $v_0,v_1 \in V$, and let $\geq_0$, $\geq_1$ be such that $C_ {\geq_0} (v_0)=C_{ \geq_1} (v_1)$. Then there exist $w \in V$,  and $\succeq_0$, $\succeq_1$ such that $v_0=C_{ \succeq_0} (w)$, $v_1 =C_{\succeq_1} (w)$, and $$\geq_0 \circ \succeq_0=\geq_1 \circ \succeq_1.$$ Moreover, the above equality holds for all such $w$, and $\succeq_0$, $\succeq_1$.
\end{proposition}

\begin{proof}
Take any $w \in V$ that fragments $v_0$, $v_1$ (there exists one by Axiom (I)), with patterns $\succeq_0$, $\succeq_1$, respectively. Put $u=C_ {\geq_0}(v_0)$, and fix  $x_u \in [u]$, $x_{v_0} \in [v_0]$, and $x _w\in [w]$ such that $x_{u} \geq_0 x_{v_0} \succeq_0 x_w$. Fix $a \in [x_w]_{\pi_w}$. Then $a \in [x_u]_u$, so there exist $x'_w \in [w]$, and $x_{v_1} \in [v_1]$ such that  $x_{u} \geq_1 x_{v_1} \succeq_0 x'_w$. But $a \in [x_w]_{\pi_w}$, so we must have $x'_w=x_w$.
\end{proof}

	\begin{remark}
		\label{re:BoundOnZII}
		Clearly, if $w \in V^Y$ is a witness of Axiom (II) for some $R$, $v$, $\bar{x}$, we can always assume that $|Y| \leq a(R)+1$. As the next proposition shows, there is an analogous bound on the size of the context of witnesses of Axiom (I).
	\end{remark}
	
	\begin{proposition}
		\label{pr:BoundOnZI}
		Let $(M,V)$ be a structure, and let $w \in V^Z$ fragment $u \in V^X$, $v \in V^Y$. Then there is a context $Z_0$, and $\geq \in \Val(Z_0,Z)$ such that $C_\geq(w)$ also fragments $u$, $v$, and $|Z_0| \leq 2^{|X|}2^{|Y|}$.
	\end{proposition}
	
	\begin{proof}
	We partition $Z$ by defininig $z \sim z'$ iff, for every $a \in [u] \cup [v]$, $$z \subseteq a \leftrightarrow z' \subseteq a.$$
	Let $Z_0=Z/\sim$, and let $\geq  \subseteq Z_0 \times Z$ be the belonging relation $\ni \uhr Z_0 \times Z$. Then $|Z_0| \leq 2^{|X|}2^{|Y|}$, and, obviously, $\geq$ is surjective and co-surjective. Moreover, since $Z_0$ is a partition of $Z$, it is also co-injective, i.e., $\geq \in \Val(Z_0,Z)$. We leave it to the reader to verify that $C_\geq(w)$ fragments both $u$ and $v$.

		%
		%

		%
		%
		
	\end{proof}

\begin{theorem}
	\label{th:SequenceDetermines}
	Let $L$ be a signature, let $X$, $Y$ be contexts, let $\geq \in \Val(X,Y)$, and let $\phi(X)$ be a formula. There is a formula $\psi(Y)$ such that $$M \vDash_{C_\geq(v)} \phi \iff M \vDash_v \psi$$ for every structure $(M,V)$ in signature $L$, and $v \in V^Y$.
\end{theorem}

\begin{proof}
First observe that, for every structure $(M,V)$ in signature $L$, $v \in V^Y$, and $\geq \in \Val(X,Y)$,
	\[ M \models_{C_\geq(v)} R(\bar{x}) \iff M \models_v \bigvee \{ R(\bar{y}): \bar{y} \in Y^{a(R)}, \, \bar{x} \geq \bar{y}  \},  \]
	so the statement of the theorem holds for atomic formulas. The rest of the proof goes by induction on the complexity of formulas. For the connectives $\neg$ and $\land$, the inductive step is straightforward.
	Fix a formula of the form $\exists^{\fr}_\succeq \phi$, where $\succeq \in \Val(X,W)$. Fix a structure $(M,V)$, and $v \in V^Y$.  If $M \vDash_{C_\geq(v)} \exists^{\fr}_\succeq \phi$, then, by Axiom (I), there is a context $Z$,  $w \in V^Z$, $\geq_0 \in \Val(Y,Z)$, and $\succeq_0 \in \Val(W,Z)$ such that $M \vDash_{C_{\succeq_0}(w)} \phi$, and $C_{\geq_0}(w)=v$. 
	Moreover, by Proposition \ref{pr:Amalgamation}, we have $$\succeq \circ \succeq_0 = \geq \circ \geq_0,$$ which means that $C_{\succeq_0}(w)$ fragments $C_{\geq \circ \geq_0}(w)=C_{\geq}(v)$ with pattern $\succeq$. 
	
	Now, we observe that, by Remark \ref{re:BoundOnZII} and  Proposition \ref{pr:BoundOnZI}, we can choose $Z$ of size that is commonly bounded for all $(M,V)$, and $v \in V^Y$. This means that there are only finitely many ways of choosing $\geq_0$, $\succeq_0$ as above. Let $\geq_i$, $\succeq_i$, $i \leq n$, enumerate all of them, and let $\psi_i$ be formulas given by the inductive assumption for $\phi$, $w$, and $\succeq_i$, $i \leq n$. 
	Then the formula $$\psi=\bigvee_{i \leq n} \exists^{\fr}_{\geq_i} \psi_i$$ is as required.
\end{proof}

Let us  introduce a new quantifier $\exists^{\re}_{\geq}$ that involves the refining relation rather than the fragmentation relation:
\[ M \vDash_v \exists^{\re}_{\geq} \phi \iff
\exists_{w \in V}[v \geq w \mbox{ and } M \vDash_w \phi]. \]
As it is more often used than $\exists^{\fr}_{\geq}$, especially in topological settings,  we will denote it  simply by $\exists_{\geq}$ (as well as write $\forall_\geq \phi$ for $\neg \exists_\geq \neg \phi$). The following corollary shows that  $\exists^{\re}_{\geq}$ is definable in the logic of co-valuations.

\begin{corollary}
	Let $L$ be a signature, let $X$, $Y$ be contexts, let $\geq \in \Val(X,Y)$, and let $\phi(Y)$ be a formula. There is a formula $\psi(X)$ such that $$M \vDash_v \exists^{\re}_\geq \phi \iff M \vDash_v  \psi$$ for every structure $(M,V)$ in signature $L$, and $v \in V^X$.
\end{corollary}

\begin{proof}
We argue as in the proof of Theorem \ref{th:SequenceDetermines}. There are contexts $Z_i$, and $\geq_i$,  $\succeq_i$, $i \leq n$, such that, for every structure $(M,V)$, we have that $v \in V^X$, and $w \in V^Y$ refines $v$ with pattern $\geq$ iff there is $i \leq n$, and $w' \in V^{Z_i}$ that fragments $v$, $w$ with patterns $\geq_i$, $\succeq_i$, respectively.  Let $\psi_i$ be given by Theorem \ref{th:SequenceDetermines} for $\geq_i$ and $\phi$, $i \leq n$. Then the formula $$\psi=\bigvee_{i \leq n} \exists^{\fr}_{\geq_i} \psi_i$$ is as required.
\end{proof}

A structure $(M,V)$ is called an $\omega$-\emph{structure} if there exists an $\omega$-chain $(v_i)$  of admissible co-valuations generating $V$, i.e., $V=\langle \{v_i\} \rangle$, i.e., every admissible co-valuation is a consolidation of some $v_i$ by a context. We will call such $(v_i)$ a \emph{core of} $V$.  As bijections preserve consolidations, we have:
\begin{remark}
\label{re:CoreIsomorphism}
If $\omega$-structures $(M,V)$, $(N,W)$ have the same signature, $(v_i)$, $(w_i)$ are cores of $(M,V)$, $(N,W)$, respectively, $f:M \to N$ is  a bijection that preserves relations, and $f[v_i]=w_i$, $i \in \NN$, then $f$ is an isomorphism of $(M,V)$ and $(N,W)$.
\end{remark}

\subsection{Compact structures}

We call a structure $(M,V)$ \emph{compact} if
	\begin{enumerate}
		\item there exists a compact $T_1$ topology $\tau_V$ on $M$ such that every $[v]$, $v \in V$, is a $\tau_V$-open cover of $M$, and every $\tau_V$-open cover of $M$ is refined by some $[v]$, where $v \in V$,
		\item relations are closed in $\tau_V$.
	\end{enumerate}
\begin{remark}
By Proposition \ref{pr:SeqGivesBasis}, the topology $\tau_V$ is generated by its basis $\bigcup_{v \in V} [v]$. In the case that $(M,V)$ is a compact $\omega$-structure, by Remark \ref{re:SeqGivesCapBasis}, this is a cap-basis.
\end{remark}

If $\tau_V$ is metrizable (connected, $0$-dimensional, etc.), we call $(M,V)$ a metrizable (connected, $0$-dimensional, etc.) structure. For compact structures $(M,V)$, $(N,W)$ in the same signature, a bijection $f:M \to N$ is a \emph{topological isomorphism} if it preserves relations, and is a homeomorphism of $(M,\tau_V)$ and $(N,\tau_W)$.

	
\begin{proposition}
\label{pr:CompactAreStructures}
Let $(M,\tau)$ be a compact $T_1$ space, let $\{R^M_\alpha\}$ be a family of closed relations on $M$, and let $\mathcal{C}$ be a family of minimal open covers of $M$ such that for any $A,B \in \mathcal{C}$ there is $C \in \mathcal{C}$ that fragments $A$, $B$, and every open cover of $M$ is refined by some $C \in \mathcal{C}$. Let $V$ be the family generated by of all co-valuations defined as in Remark \ref{re:CovertoVal}, using covers in $\mathcal{C}$. Then $(M,\{R^M_\alpha\},V)$ is a compact structure in signature $L=\{R_\alpha\}$ such that $\tau_V=\tau$. 
\end{proposition}
	
\begin{proof}
 We only need to verify Axiom (II). Fix $v \in V$, a relation $R^M_\alpha \subseteq M^m$, an $m$-tuple $\bar{x}=(x_1, \ldots, x_m)$ in $[v]$, and an $m$-tuple $\bar{a}=(a_1, \ldots, a_m)$ in $M$ such that $a_k \in x_k$ and $R^M_\alpha(\bar{a})$. 
Because each $x_k$ is open, we can find open $z'_k$ such that $a_k \in z'_k \subseteq x_k$ and $a_l \not \in z'_k$, if $a_l \neq a_k$. Fix a minimal open cover $Z'$ of the compact subspace $M \setminus \bigcup_k z_k$ consisting of sets $z'$ such that $z' \cap \{a_k\} =\emptyset$, $k \leq m$, and $Z''=Z' \cup \{z'_k\}$ refines $v$. Finally, fix $w' \in V$ that refines $Z''$, and fragments $v$. 
One can easily verify that $w=C_{\geq^{Z''}_{w'}}(w')$, $\bar{a}$ are witnesses of Axiom (II) for $v$, $\bar{x}$, $R^M_\alpha$.
\end{proof}

By Theorem \ref{th:CoinitialCapBasis}, we get
\begin{corollary}
\label{co:CompactAreStructures}
Every compact $T_1$ space with a family of closed relations can be presented as a compact structure.
\end{corollary}

\begin{corollary}
\label{co:SCCompactAreO-Structures}
Let $(M,\tau)$ be a second-countable compact $T_1$ space, let $\{R^M_\alpha\}$ be a family of closed relations on $M$, and let $\PPb_i$, $i \in \NN$, be a sequence of minimal open covers of $M$ such that every $\PPb_i$ consolidates $\PPb_{i+1}$, and every open cover of $M$ is refined by some $\PPb_i$. Then $(M,\{R^M_\alpha\},V)$, where $V=\langle \{v_i\} \rangle$, and $v_i$ are defined as in Remark \ref{re:CovertoVal}, is a compact $\omega$-structure in signature $L=\{R_\alpha\}$ such that $\tau_V=\tau$.
\end{corollary}

\begin{remark}
If $(M,V)$ is a compact $\omega$-structure in signature $L$, and $(v_i) $ is a core of $V$, then, by closedeness of the relations on $M$, we have, for every $R \in L$,
\[R^{M}(\bar{a}) \iff \forall i \in \NN \, \forall \bar{x} \in X^{a(R)}_{v_i} (\bar{a} \in [\bar{x}]_{v_i} \rightarrow M\models_{v_i} R(\bar{x})).\]
\end{remark}

Let $(M,V)$ be a structure in signature $L$, and let $(v_i) $ be an $\omega$-chain of  admissible co-valuations. Let $\PPb$ be the $\omega$-chain poset $(\bigsqcup [v_i], \supseteq)$. 
%
%
We define a compact $\omega$-stucture $(N,W)$, called \emph{the compact $\omega$-structure induced by} $(v_i)$, as follows. We put $N=\Sp \PPb$, define co-valuations $w_i \in \Val(X_{v_i},N)$ induced by $v_i$, by $$(x,S) \in w_i \iff x \in S,$$ (i.e., $w_i$ is defined essentially as in Remark \ref{re:CovertoVal}, using $[v_i]^\in$), and put $W=\langle \{w_i\} \rangle$. Finally, for $R \in L$ of arity $m$, $S_k \in N$, $k \leq m$, we define 
\[R^{N}(S_1, \ldots, S_m) \iff M \models_{v_i} R(x_1, \ldots, x_m)\]
for all $i \in \NN$, and $x_k$, $k \leq m$, such that $x_k \in S_k$. Clearly, all $R^N$ are closed in $\tau_W$. By Proposition \ref{pr:LevelMin-CovMin}, and Corollary \ref{co:SCCompactAreO-Structures}, $(N,W)$ is a compact $\omega$-structure.

\begin{remark}
\label{re:SpectraAreEnough}
If $(M,V)$ is a compact $\omega$-structure, and $(v_i)$ is a core of $V$, then, by Proposition \ref{pr:HomeoSpectrum}, Remark \ref{re:pGoesTop^in}, and Remark \ref{re:CoreIsomorphism}, the compact $\omega$-structure induced by $(v_i)$ is isomorphic with $(M,V)$.
\end{remark}

	\subsection{Properties of compact structures}

\begin{proposition}
	\label{pr:FormulaProp}	
	Let $(M,V)$ be a compact structure, let $X$ be a context, let $x_0 \in X$, $X_0 \subseteq X$, where $|X_0|>1$,  and $v \in V^X$. The following properties of $v$ can be expressed by a single formula:
	
	\begin{enumerate}
		\item $\bigcap [v]_{X_0}= \emptyset$, 
		\item $[v]$ is a partition,
		\item $[x_0]_v$ is a singleton.
	\end{enumerate}
\end{proposition}		

\begin{proof}
	
	We show (1). Define $Y=X \cup \{y_0\}$, where $y_0 \not \in X$, and $\geq \in \Val(X,Y)$ by putting $x \geq y$ iff $x=y$ or ($x \in X_0$ and $y=y_0$). Define
	\[ \text{HasEmptyIntersection}_{X_0}(X):= \forall_{\geq} \bot.\]
	Clearly, if $[v]_{X_0}$ has empty intersection, then $$M \models_v \text{HasEmptyIntersection}_{X_0}(X).$$ To see the other direction, suppose that $c=\bigcap [v]_{X_0 } \neq \emptyset$. Fix a minimal open cover $C$ of $M$ extending $\{c\}$,  and  let $w' \in V^Z$ be a fragmentation of $v$ such that $[w']$ refines $C$. Fix $y_0 \in Z$ such that $[y_0]_{w'} \subseteq c$. Note that, by minimality of $[v] $, for every $x \in X$, there is $z \in Z$ such that $z \neq y_0$, $[x]_v \supseteq [z]_{w'}$, and $[x']_v \supseteq [z]_{w'}$ implies $x=x'$.  Finally, let $\succeq \in \Val(Y,Z)$ be a pattern defined by $y \succeq z$ iff $y=z=y_0$ or ($z \neq y_0$ and  $[x]_v \supseteq [z]_{w'}$). Then $w=C_\succeq (w')$ witnesses that $M  \models_v   \exists_{\geq} \top$.

	
	In order to show (2), we define 
	\[\text{IsAPartition}(X):= \bigwedge_{\{x_1,x_2 \in X: x_1 \neq x_2\}} \text{HasEmptyIntersection}_{\{x_1,x_2\}}(X).\]
	Finally, we show (3). Let $Y=X \cup \{y_0,y_1\}$, where $y_0,y_1\ \not \in X$, $y_0 \neq y_1$.  Define $\geq \in \Val(X,Y)$ by putting $x \geq y$ iff $x=y$ or ($x=x_0$ and  ($y=y_0$ or $y=y_1$)). Let
	\[\text{IsASingleton}_{x_0}(X):=\forall_{\geq} \bot .\]
	
	Clearly, if $[x_0]_v$ is a singleton, then $$M \models_v \text{IsASingleton}_{x_0}(X).$$ To see the other direction, fix distinct $a_0, a_1 \in [x_0]_v$. Fix open $c_0, c_1 \subseteq [x_0]_v$ such that $a_0 \in c_0 \setminus c_1$, $a_1 \in c_1 \setminus c_0$,  a minimal open cover $C$ of $M$ extending $\{c_0, c_1\}$ such that $a_0, a_1 \not \in \bigcup C \setminus \{c_0,c_1\}$,  and  let $w' \in V^Z$ be a fragmentation of $v$ such that $[w']$ refines $C$. Fix $y_0, y_1 \in Y$  such that $a_0 \in [y_0]_{w'}$, $a_1 \in [y_1]_{w'}$; in particular $y_0 \neq y_1$.  Finally, let $w$ be a consolidation of $w'$ determined by $Y$ as in (1). Then $w$ witnesses that $M  \models_v   \exists_{\geq} \top$.
	%
	%
\end{proof}

For $v \in V^X$, we denote by $G(v)$ the (reflexive) graph $(X, \sqcap)$ determined by the non-empty intersection relation on $[v]$. Similarly, $S(v)$ is the finite structure in signature $L \cup \{\sqcap\}$ induced by $v$ on $X$.
\begin{remark}
	Let $X$ be a context, and let $G \subseteq X \times X$ be a reflexive graph. By Proposition \ref{pr:FormulaProp}(1), there is a single formula in the context $X$ that expresses the fact that $G(v)=G$, for every $(M,V)$, and  $v \in V^X$. We denote this formula by $\mbox{OverlapGraph}_G(X)$. 
\end{remark}


\begin{proposition}
	\label{pr:BasicSentencesSpaces}
	Let $(M,V)$ be a compact structure. The following  properties of $M$ can be expressed by a single sentence.
	
	\begin{enumerate}
		\item $|M| \geq n$, where $n \in \NN$,
		\item $M$ is disconnected,
		\item $M$ has an isolated point.
	\end{enumerate}
\end{proposition}

\begin{proof}
	For (1),  let $X$ be a context of size $n$, and  consider the sentence $\exists_X \top$. For (2), let $X$ be a context of size $2$, and consider the sentence $\exists_X \text{IsAPartition}(X)$. For (3),  let $X=\{x_0,x_1\}$, where $x_0 \neq x_1$, and consider the sentence $\exists_X \text{IsAnIsolatedPoint}_{x_0}(X)$.
\end{proof}

\begin{proposition}
	Let $(M,V)$ be a compact $\omega$-structure. The covering dimension of $(M,\tau_V)$ can be expressed by a countable theory.
\end{proposition}

%
%
%
%
%

%
%
%


\begin{proof}
	Let $X$ be a context of size $n+1$,  and define
	\[\text{PlyAtMostN}(X)= \bigwedge_{\{X_0 \subseteq X: |X_0|=n\}} \text{HasEmptyIntersection}_{X_0}(X),\]
	\[\text{DimAtMostN}_n=\forall_X \bigvee_{A \subseteq 2^X} \exists_{\geq A} \text{PlyAtMostN}(A),\]
	\[\text{DimAtMostN}=\{\text{DimAtMostN}_n: n \in \NN \}.\]
\end{proof}
%

Let $(M,V)$ be a compact $\omega$-structure, let $X$, $Y$ be contexts, let $v \in V^X$, and $w=C_\geq(v)$, for some $\geq \in \Val(X,Y)$. 
Note that it entirely depends on $G=G(w)$ and $\geq$ whether or not $[w] \lhd [v]$ . For a graph $G \subseteq X \times X$, denote the family of all corresponding $\geq \in \Val(X,Y)$ by $\mbox{Star}_G(X,Y)$.

%

\begin{proposition}
	Metrizable $\omega$-structures are exactly compact $\omega$-structures that omit all the types
	\[ A_n(X_n)=\{ \neg(\mbox{OverlapGraph}_G(X_n) \wedge  \exists_{\geq} \top): G \subseteq X_n\times X_n, \, \geq \in \mbox{Star}_G(X_n,Y) \},\]
	where $n \in \NN$, $|X_n|=n$.
\end{proposition}

\begin{proof}
	Suppose that a compact $\omega$-structure $(M,V)$ omits all the types $A_n(X_n)$. 
	As it was observed above, if for some $v \in V^X$, $G \subseteq X \times X$,  and $\geq \in \mbox{Star}_G(X,Y)$, $w \in V^Y$ witnesses that $$M \models_v \mbox{OverlapGraph}_{G}(X) \wedge  \exists_{\geq} \top,$$ then $[w] \lhd [v]$. Thus, by Proposition \ref{pr:Regular}, the topology $\tau_V$ is regular, i.e., metrizable.
	 For the other implication, use Remark \ref{re:MetrizablePosetIsRegular}.
	
\end{proof}

\paragraph{\textbf{The pseudo-arc}.} 
For a context $X$, let $\mbox{Lin}(X)$ be the set of all linear reflexive graphs on $X$. We say that a co-valuation $v \in \Val(X,M)$ is a \emph{chain-cover} if $[v]$ is a chain-cover of $M$, i.e., $G(v)$ is a linear (or chain-) graph. A continuum $X$, i.e., a connected, metrizable  compact space $X$, is \emph{chainable} if every  open cover of $X$ can be refined by an open chain-cover. Observe that if a compact $\omega$-structure $(M,V)$ is connected, $w \in V^Y$, and $C$ is a chain-cover of $M$ refined by $[w]$ with pattern $\geq \in \Val(C,Y)$, then $C_{\geq}(w)$ is also a chain-cover refining $C$.  Therefore we have:

\begin{proposition}
	Chainable continua are exactly the compact $\omega$-structures that are connected, and omit all the types $A_n(X_n)$, and $B_n(X_n)$,  $n \in \NN$, defined by	
	\[ B_n(X_n)=\{ \neg \exists_{\geq} \text{OverlapGraph}_G(Y): G \in \mbox{Lin}(Y), \, \geq \in \Val(X_n,Y) \}.\]
\end{proposition}

Finally, it is easy to define types $C_n(X_n)$, where $X_n$ is of size $n$,  such that $v \in V^{X_n}$ does not realize it iff $[v]$ is not a chain-cover or there exists $w \in V$ such that $[w]$ is a chain-cover, $[w] \lhd [v]$, and $[w]$ is crooked in $[v]$ (see \cite{BaBiVi2} for the definition). As the pseudo-arc is the unique chainable crooked continuum, we have:

\begin{proposition}
	\label{pr:Pseudoarc}
	The pseudo-arc is the unique up to homeomorphism  compact $\omega$-model of its theory that omits all the types $A_n$, $B_n$, and $C_n$, $n \in \NN$.
\end{proposition}

\paragraph{\textbf{The linearly ordered arc}.} For a binary relation $\preceq$, let us write $x\prec y$ for $(x \preceq y) \wedge \neg (y \preceq x)$. The following is easy to verify.

\begin{proposition}
	Linear orderings $(M,\preceq,V)$ on metrizable structures can be axiomatized by the following sentences: 
	
	\[ \forall_X [(x_0 \preceq x_1)\vee (x_1 \preceq x_0)],   \]
	where $X=\{x_0, x_1, x_2\}$,
	\[ \forall_X \exists_{\geq Y} [(y_1 \prec y_2)\vee (y_2 \prec y_1)],   \]
	where $X=\{x_0, x_1, x_2\}$, $Y=\{ y_0,y'_0, y_1,y'_1, y_1 \}$, $x_0 \geq y_0, y'_0$, $x_1 \geq y_1, y'_1$, $x_2 \geq y_2$, and
	\[ \forall_X [(x_0 \prec x_2)\vee \neg (x_0 \prec x_1) \vee \neg (x_1 \prec x_2) ],   \]
	where $X=\{x_0, x_1, x_2,x_3\}$.
	
\end{proposition}

It is well known that the linearly ordered arc (i.e., the unit interval) is the unique linearly ordered continuum. Thus, we get:

\begin{proposition}
	\label{pr:OrderedArc}
	The linearly ordered arc is the unique up to topological isomorphism compact $\omega$-model of its theory that omits all the types $A_n$, $n \in \NN$.
\end{proposition}

Later we will see that more can be proved about certain metrizable $\omega$-structures topologically isomorphic with the arc, the ordered arc, or with the pseudo-arc.

\section{Elements of model theory}
	
	\subsection{\L{}o\'s's theorem} Let $D $ be an ultrafiler on a set $I$, and let $(M_i,V_i)$, $i \in I$, be proto-structures in signature $L$. Let $\prod_D M_i$ be the ultraproduct of first-order structures $M_i$. For a context $X$, the ultraproduct $\prod_D v_i \in \Val(X, \prod_D M_i)$ of co-valuations $v_i \in V_i^X$ on $\prod_D M_i$ in $X$ is defined by
	\[
	(x, [(a_i)]) \in \prod_D v_i \iff (x, a_i) \in v_i \text{ for $D$-many $i$}.
	\]
	 Observe that this definition does not depend on the choice of the representative $(a_i) \in \prod M_i$. We denote the family of all ultraproducts of admissible co-valuations by $\prod_D V_i$. It is immediate to verify that
	$$ [(a_i)] \in [x]_{\pi_{\prod_D v_i}} \iff  a_i \in [x]_{\pi_{v_i}} \text{ for $D$-many  $i$},$$ 
	$$ \prod_D v_i \geq \prod_D w_i \iff v_i \geq w_i \text{ for $D$-many  $i$}, $$
	$$ C_\geq(\prod_D v_i) = \prod_D C_\geq(v_i).$$

	In particular, $(\prod_D M_i, \prod_D V_i)$ is a proto-structure. By Proposition \ref{pr:BoundOnZI}, if for proto-structures $M_i$, admissible co-valuations satisfy Axiom (I), then admissible co-valuations satisfy Axiom (I) for $(\prod_D M_i, \prod_D V_i)$. By the above observations, if $w_i$, $\bar{a}_i$ are witnesses for $R$, $v_i$, $\bar{x}$, then $\prod_D w_i$, $\prod_D \bar{a}_i$  are witnesses for $R$, $\prod_D v_i$, $\bar{x}$.  Hence,  $(\prod_D M_i, \prod_D V_i)$ is a structure, if all $(M_i,V_i)$ are structures. Thus, we have:
	
\begin{proposition}
The class of proto-structures, and the class of structures are closed under ultraproducts.
\end{proposition}
	
	\begin{theorem} [\L{}o\'s's theorem]
		Let $D$ be an ultrafilter on a set $I$,
		let $(M_i, V_i)$, $i \in I$, be a family of proto-structures in signature $L$, and let
		$v_i \in V_i$ be co-valuations in  context $X$, $i \in I$.
		For every formula~$\phi(X)$,
		\[
		\prod_D M_i \models_{\prod_D v_i} \phi
		\iff M_i \models_{v_i} \phi \text{ for $D$-many  $i$ }.
		\]
	\end{theorem}
	\begin{proof}
		The proof is by induction on the complexity of formulas.  First, observe that, for every $R \in L$, $\prod_D M \models_{\prod_D v_i} R(\bar{x})$ iff there exist $D$-many $i$, and $\bar{a}_i \in [\bar{x}]_{v_i}$ such $R^{M_i}(\bar{a}_i)$.
		The only non-trivial inductive step is $\exists^{\fr}_\geq \phi$, where $\phi$ is a formula in context $Y$, and $\geq \in \Val(X,Y)$.
		To see the direction from left to right,
		take some admissible co-valuation $\prod_D w_i \in \Val(Y,\prod_D M_i)$
		such that  $\prod_D M \models_{\prod_D w_i} \phi$,
		and $\prod_D v_i $ consolidates  $\prod_D w_i$ with pattern $\geq$.
		Since $v_i$ consolidates  $w_i$ with pattern $\geq$ for $D$-many  $i$, by the induction hypothesis,
		for $D$-many $i$, we have $M_i \models_{w_i} \phi$.
		Then we have $M_i \models_{v_i} \exists^{\fr}_\geq \phi$ for $D$-many $i$.
		To see the other direction,
		suppose $M_i \models_{v_i} \exists^{\fr}_\geq \phi$ for $D$-many $i$.
		For those $i$, take $w_i$ such that $v_i $ consolidates $w_i$ with pattern $\geq$ and $M_i \models_{w_i}\phi$,
		and let $w_i$ be an arbitrary admissible co-valuation of $M_i$ in $X$ otherwise.
		Then $\prod_D v_i$ consolidates $\prod_D w_i$ with pattern $\geq$, and $\prod_D M_i \models_{\prod_D w_i} \phi$ by the induction hypothesis,
		so we obtain $\prod_D M_i \models_{\prod_D w_i} \exists^{\fr}_\geq \phi$.
	\end{proof}
	
	A standard argument yields:
	\begin{corollary}[Compactness theorem]
		\label{co:Compactness}
		A type is consistent iff all its finite subsets are consistent. 
	\end{corollary}
	
	

	\subsection{The downward L\"owenheim-Skolem theorem}


	Let $(M,V)$, $(N,W)$ be structures. An injection $f:W \rightarrow V$ is called an \emph{embedding} if
	$$ v=C_{ \geq} (w) \iff f(v)=C_ {\geq} (f(w)).$$
	If, additionally,
	$$ M \models_w \phi \iff N \models_{f(w)} \phi$$
	for $w \in W$, then $f$ is called an \emph{elementary embedding}.

	\begin{theorem}[the downward L\"owenheim-Skolem theorem]
		\label{th:LS}
		Let $M = (M, V)$ be a structure in a countable signature $L$. For every $v_0 \in V$, there exists a compact $\omega$-structure $N=(N,W)$, and an elementary embedding $f: W \to V$ such that $v_0 \in f[W]$. In particular, $M$ and $N$ are elementarily equivalent, and for every $v_0 \in V$, there is a compact $\omega$-structure realizing the type $\tp(v_0)$.
	\end{theorem}
	
	\begin{proof}
		Fix $v_0 \in V$. Using Axioms (I) and (II), we construct an $\omega$-chain $(v_i)$ in $V$ such that, for  $V'=\langle \{v_i\} \rangle$,
		
		\begin{enumerate}
			\item $v_0 \in V'$,
			\item for all $i \in \NN$, tuples $\bar{x}$ in $X_{v_i}$, and $R \in L$ there are witnesses for $v_i$, $\bar{x}$, $R$ in $V'$,
			\item if $M \models_{v_i} \exists^{\fr}_{\geq} \phi$, then there is $j \in \NN$, and a consolidation $w$ of $v_j$ that witnesses it, i.e., $v_i=C_{ \geq} (w)$ and $M \models_{w} \phi$.
		\end{enumerate}
		 Let $\PPb$ be the $\omega$-chain poset $(\bigsqcup [v_i], \supseteq)$. Let $(\Sp \PPb,W)$ be the compact $\omega$-structure induced by $(v_i)$, where $W=\langle \{w_i\} \rangle$, and $(w_i)$ is the core induced by $(v_i)$.
		 We claim that, for every $i \in \NN$,
		\[ M \models_{v_i} \phi \iff N \models_{w_i} \phi. \]
		By Theorem \ref{th:SequenceDetermines}, this will finish the proof.
		
		Fix $R \in L$ be of arity $m$. Suppose that $N \models_{w_i} R(\bar{x})$, for $\bar{x}=(x_k)$. This means that there are $S_k \in (x_k)^\in$, $k \leq m$, such that $R^{N}((S_k))$, so $M \models_{v_i} R(\bar{x})$. On the other hand, if $M \models_{v_i} R(\bar{x})$, using Axiom (II), we fix witnesses $v \in (V')^Y$, $\bar{y} \in Y^m$, $\bar{a}$ for $v_i$, $\bar{x}$, $R$, i.e.,  $\bar{a} \in [\bar{y}]_{\pi_v}$, $R^M(\bar{a})$, and $\bar{x} \geq^{v_i}_v \bar{y}$.  
		For each $k \leq m$, we consider $S'_k=\{x \in \PPb: a_k \in x\}$. As every $[v_{i'}]$, $i' \in \NN$, is a cover of $M$, every level $\PPb_{i'}$ has non-empty intersection with $S'_k$.  Hence, by Remark \ref{re:LevelsEnough}, $S'_k$ is a selector, and, by Lemma \ref{le:SelectorFilterAndContainsMin}, there is a minimal selector $S_k \subseteq S'_k$. Let $j \in \NN$ be such that $v_j$ fragments $v$. Observe that for every $y \in [v_j]$ we have that $y \in S'_k$ implies $x_k  \geq^{v_i}_v \circ \geq^{v}_{v_j} y$; in particular, it must hold that  $S_k \in (x_k)^\in$. Finally, $\bar{a}$ ensures that for every $i' \in \NN$, if $x'_k \in S_k$, $k \leq m$, then $M \models_{v_{i'}} R((x'_k))$. Hence,  $R^N((S_k))$, and $N \vDash_{w_i} R(\bar{x})$.
		
		
		To finish the proof, we proceed by induction on the complexity of formulas. Clearly, the only non-trivial case is $\exists^{\fr}_{\geq} \phi$. Let $i \in \NN$, and $\phi$ be such that $M \models_{v_i} \exists^{\fr}_{\geq} \phi$, and let $v \in V$ witness it, i.e., $v_i=C_{ \geq} (v)$ and $M \models_v \phi$.  By (3), we can assume that $v \in V'$, i.e., there are $j \geq i$, $\geq_1$ and $\geq_2$ such that $v_i=C_{\geq_1}(v_j)$, $v=C_{\geq_2}(v_j)$. By the inductive assumption, $N \models_{C_{\geq_2}(w_j)} \phi$. By Corollary \ref{co:IsomorphicPosets}, $C_{\geq_1}(w_j)=w_i$, and $C_{\geq_1}(w_j)$ consolidates  $C_{\geq_2}(w_j)$ with pattern $\geq$, i.e., $N \models_{w_i} \exists^{\fr}_{\geq} \phi$. Using Corollary \ref{co:IsomorphicPosets} again, the other implication can be proved exactly in the same way.
		
	\end{proof}
	
	

	\subsection{The Polish space $\Mod(L)$ of infinite compact $\omega$-structures in a countable signature $L$}
	
	Let $L$ be  a countable signature. Fix pairwise disjoint contexts $X_i$, $i \in \NN$, of size $|X_0|=1$, $|X_{i+1}|=2^{|X_i|}$, and let $\PPb=\bigcup X_i$. Let $M=(M,V)$ be an infinite compact $\omega$-structure, and let $(v_i)$ be a core of $V$. We can assume that $v_i \in V^{X_i}$. Let $a \in 2^{\PPb \times \PPb}$ be the characteristic function of the $\omega$-chain ordering $\geq$ on $\PPb$ determined by $(v_i)$, i.e. $$x \geq y \iff x \geq^{v_i}_{v_j} y,$$ where $x \in X_i$, $y \in X_j$, $i \leq j$.  For a fixed $m \in \NN$, let $T_m$ be the family of all the $m$-tuples in $X_i$, $i \in \NN$. For $R \in L$, we define $b_R \in 2^{T_{a(R)}}$ by
	$$b_R(\bar{y})=1 \iff M \models_{v_i} R^M(\bar{y}),$$
	where $i \in \NN$ is unique such that $\bar{y}$ is a tuple in $X_i$. Note that because $R^M \subseteq M^m$ is closed, $b^R$ determines $R^M$. Thus, $M$ can be coded as an element $\mu_M$ of the space $2^{\PPb \times \PPb} \times \prod_{R \in L}2^{T_{a(R)}}$ with the product topology (i.e., the topology of the Cantor set).
	
	We fix an enumeration $\alpha: \NN \to C$ of the set $C$ defined by
	$$C=\{ (i,R,\bar{x}): i \in \NN, \, R \in L, \, \bar{x} \in X_i^{a(R)} \}.$$
	%
	%
	Observe that, by Remark \ref{re:BoundOnZII} and Proposition \ref{pr:BoundOnZI}, there exists a function $\beta: \NN  \to \NN$, so that the following holds. For every compact $\omega$-structure $(M,V)$ in signature $L$ there exists a core $(v_i)$ of $V$ such that, for every $m \in \NN$, witnesses of Axiom (II) for elements represented by tuples $\alpha(0), \ldots,  \alpha(m)$ can be found among consolidations of co-valuations $v_0, \ldots, v_{\beta(m)}$. We fix such $\beta$, and consider only codes of compact $\omega$-structures in signature $L$ that satisfy the above condition (this will be needed to secure an appropriate topological structure on the set of codes). To be more specific, we consider $\mu=(a,b) \in 2^{\PPb \times \PPb} \times \prod_{R \in L}2^{T_{a(R)}}$ satisfying the following conditions:
	\begin{enumerate}
		\item $a$ is the characteristic function of an ordering relation $\leq_\mu$ on $\PPb$ turning it into an $\omega$-chain $(X_i, \geq_\mu \uhr X_{i} \times X_{i+1})$ such that each $\geq_\mu \uhr X_{i+1}$ is co-bijective,
		\item for all $R \in L$, $i \in \NN$, and $\bar{y} \in X^{a(R)}_i$, if $b_R(\bar{y})=1$ and $\bar{y} \leq_\mu \bar{x}$, then $b_R(\bar{x})=1$.
			\item for all $R \in L$, $i \in \NN$, and $(x_k) \in X^{a(R)}_i$, if $b_R((x_k))=1$, then there is $j \leq 
			\beta \circ \alpha^{-1}(i,R,(x_k))$, a minimal cover $\mathcal{A}$ of $X_j$, and $(A_k) \in \mathcal{A}^{a(R)}$ such that
			\begin{enumerate}[a)]
				\item $x_k \geq_\mu x$,  for $x \in A_k$,
				\item for all $j' \geq j$, there is $(y_k) \in X^{a(R)}_{j'}$ such that $b_R((y_k))=1$, and, for all $k \leq a(R)$ and $x' \in X_{j}$, we have that  $x' \geq_\mu y_k$ implies $x'  \in  A_k$.
			\end{enumerate}
	\end{enumerate}
We define co-valuations $v_i \in \Val(X_i,M)$, $i \in \NN$, on $M=\Sp(\PPb,\leq_\mu)$ by
\[ (x,S) \in v_i \iff x \in S, \]
and relations $R^M$ on $M$, $R \in L$, by
\[ R^M((S_k)) \iff \forall (x_k) \in T_{a(R)} \, ( (\forall k \leq a(R) \, x_k  \in S_k) \rightarrow b_R((x_k))=1).  \]
Finally, we put $V=\langle \{v_i\} \rangle$, and $M_\mu=(M,V)$. Observe that all $R^M$ are closed in $\tau_V$, so, by Corollary \ref{co:SCCompactAreO-Structures}, $M_\mu$ is a compact $\omega$-structure (note that Condition (3) has not been used so far).
	
Let us define $\Mod(L)$ as the subspace of $2^{\PPb \times \PPb} \times \prod_{R \in L}2^{T_{a(R)}}$ consisting of all the elements that satisfy Conditions (1)-(3) above.  It is immediate to verify that these conditions are closed, i.e., that $\Mod(L)$ is a compact (in particular: Polish) subspace of $2^{\PPb \times \PPb} \times \prod_{R \in L}2^{T^{a(R)}}$. To sum up, we get
	
	\begin{theorem}
		\label{th:ModIsEnough}
		Let $L$ be a countable signature. Then $\Mod(L)$ is a zero-dimensional metrizable compact space, and for every compact $\omega$-structure $M$, there exists $\mu \in \Mod(L)$ such that $M_\mu$ is isomorphic with $M$.
	\end{theorem}
	
	We will also need a finer Polish topology on $\Mod(L)$. First of all, we have:
	
\begin{proposition}
\label{pr:GoodBasis}
Let $\mu=(a,b) \in \Mod(L)$, $R \in L$, $i \in \NN$, and $\bar{x} \in X^{a(R)}_i$. Then $$b(\bar{x})=1 \iff M_\mu \vDash_{v_i} R(\bar{x}).$$
	
\end{proposition}
	
	\begin{proof}
		The implication from right to left is obvious. We show the converse. For any given $i \in \NN$, $R \in L$, and tuple $\bar{x}=(x_k) \in X^{a(R)}_i$, and for $j$, $\mathcal{A}$, and $(A_k)$ as in Condition (3) above, we define $\succeq \in \Val(\mathcal{A}, X_j)$ by 
		$$ A \succeq x \iff x \in A.$$
		We put $v=C_{\succeq}(v_j)$. By (3a), we have $[x_k]_{v_i} \supseteq \bigcup_{a \in A_k} [a]_v$. 
		Using (3b), we construct, exactly as in the proof of Proposition \ref{pr:LevelMin-CovMin}, $x^j_k \in X_{i+j}$, $j \in \NN$, $k \leq a(R)$ such that, $x^0_k=x_k$, $b_R((x^j_k))=1$, and, for  $x' \in X_j$, $x' \geq_\mu x^{j+1}_k$ implies $x'=x^j_k$.  Then, since, by Condition (1), $(\PPb,\leq_\mu)$ is an $\omega$-chain, we get that, for $x' \in X_i$, $x' \geq_\mu x^k_{j+1}$ implies $x'=x^0_k=x_k$. In particular, every minimal selector $S_k \subseteq \{x^j_k\}^{\leq_\mu}$ is an element of $[x_k]_{v_i}$. By the definition of $M_\mu$, we get $R^M((S_k))$, so $M_\mu \vDash_{v_i} R(\bar{x})$.
		
		
	\end{proof}

	Proposition \ref{pr:GoodBasis} implies that the sets $[i,\geq]$, and $[i,R,\bar{x}]$, where $i \in \NN$, $\geq$ is a binary relation on $\bigcup_{i' \leq i} X_{i'}$, $R \in L$, and $\bar{x} \in X^{a(R)}$, defined by
	\[ [i,\geq]=\{ \mu \in \Mod(L): (x \geq_\mu y \leftrightarrow x \geq y), \, x, y \in \bigcup_{i' \leq i} X_{i'}\}, \]
	\[ [i, R,\bar{x}]=\{\mu \in \Mod(L): M_\mu \vDash_{v_i} R(\bar{x}) \},\]
	form a sub-basis for the topology on $\Mod(L)$ inherited from $2^{\PPb \times \PPb} \times \prod_{R \in L}2^{T_{a(R)}}$. By standard arguments as in the proof of \cite[Theorem 11.4.1]{Gao}, this topology can be refined to a Polish topology given by sub-basic open sets
	\[ [i,\geq], \]
	\[ [i,\phi]=\{\mu \in \Mod(L): M_\mu \vDash_{v_i} \phi \},\]
	where $\phi$ is a formula. We call it the \emph{logic topology} on $\Mod(L)$. For a theory $T$ in signature $L$, we denote by $\Mod(T)$ the closed (in the logic topology), and thus Polish, subspace of $\Mod(L)$ consisting of all $\mu \in \Mod(L)$ such that $M_\mu$ is a model of $T$.

	\subsection{Omitting types, isolated types, atomic models.}
	
	For a theory $T$ in a countable signature $L$, $n \in \NN$, and a fixed context $X$ of size $n$, the type space $S_n(T)$ is the set of all complete types $p(X)$ realized by models of $T$. We endow $S_n(T)$ with the standard logic topology, given by open neighborhoods $$[\phi]=\{ p \in S_n(T): \phi \in p \},$$ where $\phi$ is a formula. Clearly, $S_n(T)$ is zero-dimensional, and, by Corollary \ref{co:Compactness}, it is compact.
	
\begin{theorem}
\label{th:OM}
Let $T$ be a theory in a countable signature. Assume that, for some $n \in \NN$, $F \subseteq S_n(T)$ is meager (in the logic topology). Then, in the logic topology on $\Mod(T)$, there are comeagerly many $\mu \in \Mod(T)$ such that $M_\mu$ omits every type in $F$.
\end{theorem}
	
\begin{proof}
For $i \in \NN$, and $\succeq \in \Val(X,X_i)$, let $f_{\succeq}:\Mod(T) \rightarrow S_n(T)$ be the mapping defined by $f_{\succeq}(\mu)=\tp_{M_\mu}(C_{\succeq}(v_i))$. We  clearly have $$f_{\succeq}([i, \psi]) \subseteq [\phi],$$ for every neighborhood $[\phi]$, where $\psi$ is given by Theorem \ref{th:SequenceDetermines} for $X_i$, $X$, $\succeq$ and $\phi$. In other words, $f_{\succeq}$ is continuous in the logic topology on $\Mod(T)$.
		Moreover, an inspection of the proof of Theorem \ref{th:LS} gives that $$f_{\succeq}([j, \geq] \cap [j, \phi]) \supseteq [\exists^{\fr}_{\geq \circ \succeq} \phi],$$
		for every $j \geq i$. Thus, $f_{\succeq}$ is category-preserving, and so $E=\bigcup_{ \succeq} f^{-1}_{\succeq}[F]$ is meager. 
		Obviously, $M_\mu$ omits every type in $F$,  for $\mu \in \Mod(T) \setminus E$.
	\end{proof}
	
	
	\begin{theorem}
		\label{th:OM-TFAE}
		Let $T$ be a theory in a countable signature, and let $p \in S_n(T)$. The following are equivalent:
		\begin{enumerate}
			\item $p$ is isolated (in the logic topology),
			\item $p$ is realized in every model of $T$,
			\item there exist $i \in \NN$ and $\geq$  such that $\{\mu \in \Mod(T): M_\mu \models_{C_\geq(v_i)} p \}$ has non-empty interior.
		\end{enumerate}
	\end{theorem}
	
	\begin{proof}
		(1)$\Rightarrow$(2) is obvious, and (2)$\Rightarrow$(1), (3)$\Rightarrow$(1) follow from Theorem \ref{th:OM}. To show that (2)$\Rightarrow$(3), observe that the set $F \subseteq \Mod(T)$ consisting of all $\mu \in \Mod(T)$ that realize $p$ is an $F_\sigma$ subset of $\Mod(T)$. Indeed, $F=\bigcup_{\geq,i} F_{\geq,i}$, where $F_{\geq,i}$ is the set of $\mu \in \Mod(T)$ such that $C_\geq(v_i)$ realizes $p$. Therefore, if $F=\Mod(T)$, there must exist  $\geq$ and $i \in \NN$ such that $\{\mu \in \Mod(T): M \vDash_{C_\geq(v_i)} p \}$ is non-meager. As this set is closed, it must have non-empty interior.
	\end{proof}
	
As usual, we say that a type $p(X)$ of a theory $T$ \emph{isolated} if there is a formula $\phi(X)$ such that $M \models_v \phi$ implies $M \models_v \psi$, $\psi \in p$. By Theorem \ref{th:OM-TFAE}, $p \in S_n(T)$ is isolated iff it is an isolated point in $S_n(T)$.

\begin{corollary}[Omitting types]
	\label{co:OM}
		Let $T$ be a theory in a countable signature, and let $p(X)$ be a non-isolated $n$-type of $T$. Then the subset of $S_n(T)$ of all $r \in S_n(T)$ such that $p \subseteq r$, is meager in the logic topology. In particular, every countable family $F$ of non-isolated types of $T$ is omitted by some compact $\omega$-model of $T$. 
	\end{corollary}
	
\begin{proof}
Clearly, the subset $A \subseteq S_n(T)$ as in the statement of the corollary is closed in the logic topology. If $A$ is not meager, there is a formula $\phi$ such that $[\phi] \subseteq A$, i.e., $p$ is isolated; a contradiction.
\end{proof}
	
A model of a theory $T$ is \emph{atomic} if it is a compact $\omega$-structure realizing only isolated types of $T$. A compact $\omega$-structure is atomic if it is an atomic model of its theory. 

\begin{remark}
\label{re:CoreAtomicIsAtomic}
By Theorem \ref{th:SequenceDetermines}, a compact model of a theory $T$ is atomic iff it has a core consisting of co-valuations with isolated types.  
\end{remark}	

\begin{corollary}
\label{co:MiniSuzuki}
Let $T$ be a theory in a countable signature. Suppose that there exists a countable family $F$ of types such that any two compact $\omega$-models of $T$ that omit $F$  are isomorphic. Then every such model is an atomic model of $T$.
\end{corollary}

\begin{proof}
Fix $M$ as in the statement of the corollary, and suppose that it realizes a non-isolated type $p$. By Corollary \ref{co:OM}, there is a compact model $N$ of $T$ that omits $F \cup \{p\}$. But then $M$ is isomorphic to $N$, a contradiction.
\end{proof}


	

	\begin{theorem}
		A theory $T$ in a countable signature has an atomic model iff isolated types are dense in $S_n(T)$ for every $n \in \NN$.
	\end{theorem}
	
	\begin{proof}
		The implication from left to right is obvious. The converse follows from Theorem \ref{th:OM} because then non-isolated types form a meager subset of $S_n(T)$ for every $n \in \NN$.
	\end{proof}
	
	\begin{theorem}
		\label{th:UniqueAtomic}
		Let $T$ be a theory in a countable signature, let $(M,V)$, $(N,W)$ be atomic models of $T$, and suppose that $v \in V$, $w \in W$ have the same  type. Then there is an isomorphism of $M$ and $N$ that maps $v$ to $w$. 
	\end{theorem}
	
\begin{proof}
	By Remark \ref{re:SpectraAreEnough}, we can assume that $M= \Sp \PPb$, $N=\Sp \RRb$, where $\PPb$, $\RRb$ are $\omega$-chains in $V$, $W$ induced by cores $(v'_i)$, $(w'_i)$, respectively. We  construct sequences of co-valuations $v_i \in V^{X_i}$, $w_i \in W^{X_i}$ with the same type. Let $v_0=v$, $w_0=w$, and suppose that $v_i$, $w_i$, $i \leq n$, have been already constructed for some  $n=2k$. Let $v_{n+1} \in V^{X_{n+1}}$ be an admissible co-valuation fragmenting $v_n$ and $v'_{k}$. Then there exists $w_{n+1} \in W^{X_{n+1}}$ that fragments $w_n$ with pattern $\geq^{v_n}_{v_{n+1}}$, and that has the same type as $v_{n+1}$. Analogously, switching the roles of $v_n$, $w_n$, we construct $w_{n+1}$, $v_{n+1}$, for odd $n$. Clearly, $(v_i)$, $(w_i)$ are cores of $V$, $W$, respectively, they induce the same ordering on $\bigsqcup X_i$, and $M \models_{v_i} R(\bar{x})$ iff $N \models_{w_i} R(\bar{x})$, for $R \in L$. By Remark \ref{re:SpectraAreEnough}, the mapping $v_i \mapsto w_i$ induces an isomorphism of $M$ and $N$. 
\end{proof}
	
	
	
We say that a theory is \emph{$\omega$-categorical} if it has a unique up to isomorphism compact $\omega$-model.

\begin{theorem}[a Ryll-Nardzewski theorem]
		Let $T$ be a theory in a countable signature. The following are equivalent.
		
		\begin{enumerate}
			\item $T$ is  $\omega$-categorical,
			\item every type of $T$ is isolated,
			\item $S_n(T)$ is finite for every $n \in \NN$.
		\end{enumerate}
	\end{theorem}
	
	\begin{proof}
		The equivalence (2)$\Leftrightarrow$(3), as well as (1)$\Rightarrow$(2) follow from Theorem \ref{th:OM}. The implication (2)$\Rightarrow$(1) follows from Theorem \ref{th:UniqueAtomic}.
	\end{proof}
	
\section{\fra \  limits.}

Let us briefly present an extension of  \fra  \ theory developed in Sections 3 and 4.5 of \cite{BaBiVi2}. The reader is referred to this paper for more details on the original treatment.

For sets $A$, $B$, and a relation $R$ of arity $m$ defined both on $A$, and on $B$, we say that a relation $\sqsp \subseteq A \times B$  is \emph{$R$-preserving} if, for every $\bar{b} \in B^m$, we have that $R(\bar{b})$ and $\bar{a} \sqsp \bar{b}$ implies $R(\bar{a})$. It is \emph{$R$-surjective} if, for every $\bar{a} \in A^m$, $R(\bar{a})$ implies that there is $\bar{b} \in B^m$ such that $R(\bar{b})$, and $\bar{a} \sqsp \bar{b}$.  For graphs $(G,\sqcap)$, $(H,\sqcap)$, we say that $\sqsp \subseteq G \times H$ is \emph{edge-witnessing} if, for all $g, g' \in G$, $g \sqcap g'$ implies that there is $h \in H$ with $g,g' \sqsp h$. 

Let $L$ be a countable relational signature. Let $\SSf_L$ be the class of all finite structures in $L \cup \{\sqcap\}$, where $\sqcap$ is a binary relation symbol, such that $\sqcap^A$ is a reflexive graph, for every $A \in \SSf_L$.  We regard $\SSf_L$ as a category with morphisms $\sqsp$ from $B$ to $A$, $A,B \in \SSf_L$, defined as all relations  $\sqsp \subseteq A \times B$ that are are co-bijective, and  $R$-preserving, $R \in L \cup \{\sqcap\}$; composition of morphisms in $\SSf_L$ is composition of relations. In this section, we will consider only subcategories $\KKf \subseteq \SSf_L$ that are closed under isomorphism, and such that edge-witnessing morphisms are \emph{wide}, i.e., for every $A \in \KKf$, there is $B \in \KKf$, and an edge-witnessing morphism from $B$ to $A$.

\begin{remark}
The role of the graph relation $\sqcap$ is to control the structure of $\omega$-posets induced by sequences in subcategories of $\SSf_L$, as  discussed in \cite{BaBiVi2}. In particular, the category $\SSf_\emptyset$ is the category $\BBf$ described in Section 3 of \cite{BaBiVi2}.   

\end{remark}


Let  $\KKf \subseteq \SSf_L$ be a category. We write $\KKf^A_B$ for the collection of all morphisms in $\KKf$ from $B$ to $A$. We say that $\KKf$ has \emph{amalgamation} if for any $A,B_0,B_1 \in \KKf$, and $\sqsp_0 \in \KKf^A_{B_0}$, $\sqsp_1 \in \KKf^A_{B_1}$, there are $D \in \KKf$, and $\sqni_0 \in \KKf^{B_0}_D$, $\sqni_1 \in \KKf^{B_1}_D$ such that $\sqsp_0 \circ \sqni_0=\sqsp_1 \circ \sqni_1$. A sequence $(C_i, \sqsp_i)$ in $\KKf$ is \emph{co-initial} if for every $C \in \KKf$ there are $i \in \NN$ and $\sqsp \in \KKf^C_{C_i}$. It is \emph{absorbing} if for every $i \in \NN$, $C \in \KKf$, and $\sqsp \in \KKf^{C_i}_C$ there are $j \in \NN$, and $\sqni \KKf^C_{C_j}$ such that $\sqsp \circ \sqni=\sqsp^i_j$. It is \emph{\fra} if it is co-initial, and absorbing.  One can easily show (see, e.g., \cite[Proposition 4.30]{BaBiVi2}) that $\KKf$ is directed and has amalgamation iff there exists a \fra \  sequence in $\KKf$.

As before, every sequence $(\KKb_i, \sqsp_i)$ in $\KKf$ induces an  $\omega$-chain poset $\KKb=(\bigsqcup \KKb_i, \sqsp)$. We can consider the induced compact $\omega$-structure $(\Sp \KKb, V )$ in signature $L$, where $V$ is generated by co-valuations $v_i \in \Cov(\KKb_i, M)$ determined by covers $\KKb^\in_i$, $i \in \NN$, and, for $R \in L$ of arity $m$, we define

\[R^{M}(S_1, \ldots, S_m) \iff R^{\KKb_i}(x_1, \ldots, x_m)\]
for all $i \in \NN$, and $x_k$, $k \leq m$, such that $x_k \in S_k$. Clearly, all $R^M$ are closed in $\tau_V$. We call $(\KKb_i, \sqsp_i)$ \emph{$L$-faithful} if the structure $S(v_i)=\KKb_i$, $i \in \NN$. If $(\KKb_i, \sqsp_i)$ is a \fra \ sequence, we call $( \Sp \KKb,V)$ a  \emph{\fra \  limit} of $\KKf$. 

Note that because edge-witnessing morphisms are assumed to be wide in $\KKf$, and, by Proposition \ref{pr:LevelMin-CovMin}, $p^\in \neq \emptyset$, for $p \in \KKb$, we have that the graph $G(v_i)$ is equal to the reduct of $\KKb_i$ to $\{\sqcap\}$ for every \fra \ sequence  $(\KKb_i, \sqsp_i)$. The following fact can be proved in the same way as Proposition \ref{pr:LevelMin-CovMin}.

\begin{proposition}
\label{pr:PiRSurjective}
Let $L$ be a countable signature, and let $\KKf \subseteq \SSf_L$ be a category. If $\pi_{\sqsp}$ is $R$-surjective for all morphisms $\sqsp$ in $\KKf$,  and $R \in L$,  then every sequence in $\KKf$ is $L$-faithful.
\end{proposition}

It is known (see \cite[Proposition 4.32]{BaBiVi2}) that any two \fra \ sequences in a category $\KKf \subseteq \SSf_\emptyset$ have homeomorphic spectra. It turns out that more can be said about them in the context of compact $\omega$-structures. The next two theorems indicate that for a second-countable compact $T_1$ space $X$ equipped with closed relations in signature $L$ that is topologically isomorphic to a \fra \ limit of a category $\KKf \subseteq \SSf_L$,  the (unique up to isomorphism) $\omega$-structure induced by an $L$-faithful  \fra \ sequence in $\KKf$ can be thought of as a canonical $\omega$-structure associated with $X$. 

\begin{remark}
Observe that if $w=C_\geq(v)$, then $\geq$ is a morphism from $S(v)$ to $S(w)$ in $\SSf_L$, i.e., it preserves all $R \in L$, and the non-empty intersection graph relation. 
\end{remark}

\begin{lemma}
	\label{le:TypesIso}
Let $L$ be a countable signature. Let $(M,V)$ be a \fra \ limit of an $L$-faithful \fra \  sequence $(\KKb_i, \sqsp_i)$ in a category $\KKf \subseteq \SSf_L$. Then for every $v \in V$ such that $S(v) \in \KKf$, $D \in \KKf$,  and $\sqsp \in \KKf^{S(v)}_{D}$, there are $j \in \NN$, and $\sqni \in \KKf^{D}_{\KKb_j}$ such that $C_{\sqni} (\KKb^\in_j)$ fragments $v$ with pattern $\sqsp$.
\end{lemma}

\begin{proof}
	For the case that $v=\KKb_i^\in$, this follows from Corollary \ref{co:IsomorphicPosets}, i.e., from the fact that $\omega$-chains $(\bigsqcup \KKb_i, \sqsp)$ and $(\bigsqcup \KKb^\in_i, \supseteq)$  are isomorphic posets, and the assumption that $(\KKb_i, \sqsp_i)$ is $L$-faithful. Otherwise, we put $\geq_0=\sqsp$, and select $i \in \NN$ such that $\KKb^\in_i$ fragments $v$ with pattern $\geq_1$. Then we use amalgamation in $\KKf$ to obtain an amalgam $E \in \KKf$ of $v$, $\geq_0$, $\geq_1$, together with appropriate $\succeq_0$, $\succeq_1$. Finally, we apply the above case to $\KKb^\in_i$ and $\succeq_1$ to find $j \in \NN$ and $\sqni$ such that $S(C_{\sqni}(\KKb^\in_j))=E$. Because of the way $E$ has been chosen,  $S(C_{\succeq_0 \circ \sqni}(\KKb^\in_j))=D$, and $C_{\succeq_0 \circ \sqni}(\KKb^\in_j)$ fragments $v$ with pattern $\sqsp$. 
	
\end{proof}

\begin{theorem}
	\label{th:FraAreIso}
	Let $L$ be a countable signature.  Let $(M,V)$, $(M',V')$ be \fra \ limits of $L$-faithful \fra \  sequences $(\KKb_i, \sqsp_i)$, $(\KKb'_i, \sqsp'_i)$ in a category $\KKf \subseteq \SSf_L$. Then $(M,V)$, $(M',V')$ are isomorphic. In other words, their exists a unique \fra \ limit of an $L$-faithful  \fra \  sequence in $\KKf$.
\end{theorem}

\begin{proof}
We construct isomorphic cores $(v_{i})$, $(v'_{i})$ of $V$, $V'$ as follows. Let $v_0=\{M\}$, $v'_0=\{M'\}$, and suppose that $v_i$, $v'_i$, $i \leq n$, have been constructed for some $n=2k$. Let $v_{n+1} \in V$ be a co-valuation that fragments $\KKb^\in_{k}$, and fragments $v_n$ with pattern $\sqsp \in \KKf^{S(v_n)}_{S(v_{n+1})}$. There is such $v_{n+1}$ because $\KKf$ is directed. By Lemma \ref{le:TypesIso},  there is $v'_{n+1} \in W$ with $S(v'_{n+1})=S(v_{n+1})$ that fragments $v'_n$ with pattern $\sqsp$. For odd $n$, we proceed analogously, switching the roles of $v_n$ and $v'_n$.
\end{proof}

It is straightforward to construct a formula $\mbox{Structure}_K(X)$ that says that $S(v)$, where $v \in \Val(X,M)$, is isomorphic to a given structure $K$.

\begin{theorem}
\label{th:FraAtomic}
	Let $L$ be a countable signature. 
	Assume that $(\KKb_i, \sqsp_i)$ is an $L$-faithful \fra \  sequence in a category $\KKf \subseteq \SSf_L$. Then its \fra \  limit is an atomic $\omega$-structure. 
\end{theorem}

\begin{proof}
Let $D_n$, $n \in \NN$, be types defined by
\[ D_n(X_n)=\{ \neg \exists^{\fr}_{\geq} \text{Structure}_K(Y): K \in \KKf, \, \geq \in \Val(X_n,Y) \}.\]
Let $(N,W)$ a compact $\omega$-model of the theory $T$ of the \fra \ limit $(M,V)$ of $(\KKb_i, \sqsp_i)$ that omits all the types $D_n$. Then co-valuations $w \in W$ such that $S(w) \in \KKf$ are co-initial in $W$.  We observe that the statement of Lemma \ref{le:TypesIso} can be expressed as an infinite conjunction of sentences in the logic of co-valuations, i.e., it is implied by $T$. Now, one easily constructs a core $(w_i)$ of $W$ such that $(S(w_i), \geq^{w_i}_{w_{i+1}})$ is isomorphic with $(\KKb_i, \sqsp_i)$. By Remark \ref{re:SpectraAreEnough}, $(N,W)$ is isomorphic with $(M,V)$. Finally, we apply Corollary \ref{co:MiniSuzuki}.

\end{proof}

Let $\IIf \subseteq \SSf_\emptyset$ be the category of all linear  graphs with monotone morphisms, i.e., morphisms $\sqsp \subseteq G \times H$ such that $A^\sqsp$ is connected, provided that $A$ is connected. Let $\PPf \subseteq \SSf_\emptyset$ be the category of all linear graphs with all morphisms. Let $\FFf \subseteq \SSf_\emptyset$ be the category of all fans, i.e., trees with exactly one vertex of degree $\geq 3$, called the root, and spoke-monotone morphisms, i.e., morphisms that are monotone on every maximal path with the root as an end-point.  All these categories are directed and have amalgamation (see \cite[Proposition 5.13]{BaBiVi2},\cite[Theorem 4.14]{BiMa}, and  \cite[Proposition 5.29]{BaBiVi2}), and their respective \fra \  limits are homeomorphic to the arc, the pseudo-arc, and the Lelek fan. Thus, we get
 
\begin{corollary}
The arc, the pseudo-arc, and the Lelek fan can be presented as atomic $\omega$-structures.
\end{corollary}

\begin{remark}
The proof of Theorem \ref{th:FraAtomic} gives that \fra \ limits are $\omega$-categorical relative to the class of structures that can be obtained from the respective category. For example, the atomic model of the pseudo-arc is the unique chainable $\omega$-compact model of its theory. Also, Proposition \ref{pr:OrderedArc} says that the ordered arc is the unique metrizable $\omega$-compact model of its theory.
\end{remark}

As a matter of fact, we can explicitely characterize $\omega$-structures homeomorphic with the arc that are  \fra \  limits of $\IIf$.  Recall that a \emph{taut} cover of a topological space $X$ is a cover $C$ such that  $\cl{c} \cap \cl{c'} \neq \emptyset$ implies $c \cap c' \neq \emptyset$, for $c,c' \in C$. We will say that a co-valuation $v$ is connected if every $x \in [v]$ is connected. In particular, connected open co-valuations on the arc are consist of open intervals. 
%


\begin{proposition}
\label{pr:ExcAreFra}
Let $(I,V)$ be a compact $\omega$-structure homeomorphic with the arc, and let $(v_i)$ be a core of $V$. The sequence $(G(v_i), \geq^{v_i}_{v_{i+1}})$ is a \fra \  sequence in $\IIf$ iff $(v_i)$ consists of connected taut chain-covers.
\end{proposition}

\begin{proof}
Suppose that $(v_i)$ consists of connected taut chain-covers. Fix $v_i \in V^{X}$, and let $\geq \in \Val(X,Y)$ be monotone. It is an easy excersise to prove that there exists a connected taut chain-cover $C$ (not necessarily in $V$) that fragments $v$ with pattern $\geq$. Fix $v_j \in V$, where $j \geq i$, that refines $C$ with maximal pattern $\geq^C_{v_j}$, and fragments $v$. Consider $w=C_{\geq^C_{v_j}}(v_j)$. It obviously refines $v_i$ with  pattern $\geq$. We show that it also fragments $v_i$ with this pattern. Fix $z=(a_z,b_z) \in [w]$, and $x=(a_x,b_x) \in [v_i]$, $y=(a_y,b_y) \in C$  such that $x \supseteq z$, $y \supseteq z$. Suppose that $x \not \supset y$, i.e., $$a_y\leq a_z<b_z \leq b_x<b_y$$ or $$a_y <a_z \leq a_x< b_x \leq b_y.$$ We consider only the former case, the other one is analogous. Since $C$ fragments $[v_i]$, there must be $y'=(a_{y'},b_x) \in C$ with $$a_x \leq a_{y'}<a_y,$$ i.e., $x \supseteq y' \supseteq z$, and $x \geq y' \geq^C_w z$. As $w$ fragments $v_i$, this shows that every $x \in [v]$ is a union of $y \in C_{\geq^C_w}(w)$ with $x \geq y$. Thus, $(G(v_i), \geq^{v_i}_{v_{i+1}})$  is a \fra \  sequence.

The converse follows from the fact that any two \fra \  sequences have isomorphic limits.
\end{proof}

Finally, let $\LLf \subseteq \SSf_{\leq}$ be the category of all linear graphs with linear orders compatible with the graph relation. Using Proposition \ref{pr:PiRSurjective}, one immediately sees that every \fra \ sequence in $\LLf$ is $<$-faithful. Arguing as for $\IIf$, we get:

\begin{proposition}
Let $(I,\leq, V)$ be a compact $\omega$-structure topologically isomorphic with the ordered arc, and let $(v_i)$ be a core of $V$. The sequence $(S(v_i), \geq^{v_i}_{v_{i+1}})$ is a \fra \  sequence in $\LLf$ iff $(v_i)$ consists of connected taut chain-covers.

\end{proposition}


\begin{thebibliography}{99}
\bibitem{Ba} P. Bankston, \emph{Base-free formulas in the lattice-theoretic study of compacta}, Archive Math. Logic 50 (2011), 531--542.  
\bibitem{BaBiVi} A. Barto\v s, T. Bice, A. Vignati,  \emph{Constructing compacta from posets}, Publ. Mat. 69 (2025), 217--265.
\bibitem{BaBiVi2} A. Barto\v s, T. Bice, A. Vignati, \emph{Generic compacta from relations between finite graphs: Theory building and examples}, arXiv:2408.15228.
\bibitem{BiMa} T. Bice, M. Malicki, \emph{Homeomorphisms of the pseudoarc}, arXiv:2412.20401.
\bibitem{EaGoVi} C. J. Eagle, I. Goldbring, A. Vignati, \emph{The pseudoarc is a co-existentially closed continuum}, Top. App. 207 (2016), 1--9.
\bibitem{FlZi} J. Flum, M. Ziegler, \emph{Topological Model Theory}, Lecture Notes in Mathematics, Springer 1980. 
\bibitem{Gao} S. Gao, \emph{Invariant Descriptive Set Theory}, Chapman and Hall 2008.
\bibitem{Gr} A. Grzegorczyk,  \emph{Undecidability of Some Topological Theories}, Fund. Math. 38 (1951), 137--152.
\bibitem{Gu} Y. Gurevich, \emph{Monadic theory of order and topology, I}, Israel J. Math. 27 (1977), 299--319.
\bibitem{Heatal}  C.W. Henson, C.G. Jockusch, Jr. L.A. Rubel, G. Takeuti,  \emph{First order topology}, Dissertationes Mathematicae 143, 1977. 
\bibitem{MKTa} J.C.C. McKinsey, A. Tarski,  \emph{The Algebra of Topology}, Annals Math. 45 (1944), 141--191.
\bibitem{Pi} A. Pillay, \emph{First order topological structures and theories}, J. Symb. Logic 52 (1987), 763--778.
\bibitem{Ta} A. Tarski, \emph{Der Aussagenkalkül und die Topologie}, Fund. Math. 31 (1938), 103--134.
\bibitem{Wa} H. Wallman, \emph{Lattices and topological spaces}, Annals Math.  39 (1938), 112--126.

\end{thebibliography}
\end{document}